\documentclass[reqno]{amsart}

\usepackage[alphabetic,initials,nobysame]{amsrefs}
\usepackage{amssymb,mathtools,mathrsfs,enumitem,color,comment}
\usepackage{pgfplots} \pgfplotsset{compat=1.18}

\usepackage[bookmarksdepth=2]{hyperref}
\BibSpec{collection.article}{%
	+{}  {\PrintAuthors}				{author}
	+{,} { \textit}                  	{title}
	+{.} { }                         	{part}
	+{:} { \textit}                  	{subtitle}
	+{,} { \PrintContributions}      	{contribution}
	+{,} { \PrintConference}         	{conference}
	+{}  {\PrintBook}                	{book}
	+{,} { }                         	{booktitle}
	+{,} { }						 	{series}
	+{,} { }						 	{publisher}
	+{,} { }						 	{address}
	+{,} { \PrintDateB}              	{date}
	+{,} { pp.~}                     	{pages}
	+{,} { }                         	{status}
	+{,} { \PrintDOI}                	{doi}
	+{,} { available at \eprint}     	{eprint}
	+{}  { \parenthesize}            	{language}
	+{}  { \PrintTranslation}        	{translation}
	+{;} { \PrintReprint}            	{reprint}
	+{.} { }                         	{note}
	+{.} {}                          	{transition}
	+{}  {\SentenceSpace \PrintReviews} {review}
}

\theoremstyle{plain}
\newtheorem{theorem}{Theorem}
\newtheorem{lemma}[theorem]{Lemma}

\newtheorem{corollary}[theorem]{Corollary}
\newtheorem{question}[theorem]{Question}
\newtheorem{problem}[theorem]{Problem}
\newtheorem{conjecture}[theorem]{Conjecture}

\theoremstyle{definition}
\newtheorem{definition}[theorem]{Definition}

\theoremstyle{remark}
\newtheorem{remark}[theorem]{Remark}
\newtheorem{example}[theorem]{Example}

\numberwithin{equation}{section}
\numberwithin{theorem}{section}
\newcommand{\br}{\overline}
\newcommand{\R}{\mathbb R}

\newcommand{\C}{\mathbb C}
\newcommand{\D}{\mathbb D}

\newcommand{\N}{\mathbb N}

\DeclareMathOperator{\dist}{{\mathrm{dist}}}
\DeclareMathOperator{\diam}{{\mathrm{diam}}}

\DeclareMathOperator{\inter}{{\mathrm{int}}}

\DeclareMathOperator{\Mod}{Mod}

\DeclareMathOperator{\loc}{\mathrm{loc}}

\DeclareMathOperator{\llc}{\mathit{LLC}}
\title{Uniformization of metric surfaces: A survey}
\author{Dimitrios Ntalampekos}
\address{Department of Mathematics, Aristotle University of Thessaloniki, Thessaloniki, 54152, Greece.}
\email{dntalam@math.auth.gr}
\date{\today}
\keywords{Uniformization, Riemann surface, metric surface, quasisphere, quasisymmetric, conformal, quasiconformal, weakly quasiconformal, modulus, reciprocal}
\subjclass[2020]{Primary 30C62, 30C65, 30F10, 30L10, 53A05; Secondary 28A75, 51F99, 53C23, 53C45}
\dedicatory{Dedicated to Mario Bonk.}
\begin{document}

	\begin{abstract}
	In this survey we present the most recent developments in the uniformization of metric surfaces, i.e., metric spaces homeomorphic to two-dimensional topological manifolds. We start from the classical conformal uniformization theorem of Koebe and Poincar\'e. Then we discuss the Bonk--Kleiner theorem on the quasisymmetric uniformization of metric spheres, which marks the beginning of the study of the uniformization problem on fractal surfaces. The next result presented is Rajala's theorem on the quasiconformal uniformization of metric spheres. We conclude with the final result in this series of works, due to Romney and the author, on the weakly quasiconformal uniformization of arbitrary metric surfaces of locally finite area under no further assumption.
	\end{abstract}

\maketitle

\setcounter{tocdepth}{1}
\tableofcontents

\section{Introduction}

\subsection{Motivation}
The object of this survey is to present the most recent developments in the uniformization of surfaces, starting from the classical conformal uniformization theorem of Koebe and Poincar\'e and concluding with the uniformization of arbitrary non-smooth metric surfaces of locally finite area. 

One way to understand the geometry of a fractal metric space is to transform it, as if it were made out of rubber, to a standard, well-understood smooth space, such as a circle or a sphere. The transformation should be well-behaved, in the sense that it should preserve the geometry. Examples of such transformations are conformal (angle-preserving), quasiconformal (angle-quasipreserving), quasisymmetric (shape-quasipreserving), and bi-Lipschitz (shape-and-scale-quasipreserving) maps and depending on the setting we choose the appropriate type of map. For example, for surfaces of finite area all mentioned types of maps are appropriate, but for surfaces of very fractal nature only quasisymmetric maps seem to be applicable. Having such transformations allows one to study a metric space using classical tools of Euclidean space, such as first-order calculus and potential theory. Whether the transformation of a metric space to a smooth or standard space is possible is the object of the problem of \textit{uniformization}. 

The one-dimensional version of the uniformization problem is very well understood. For example, every curve of finite length admits a Lipschitz parametrization by an interval. As for fractal curves, Ahlfors \cite{Ahlfors:reflections} characterized those Jordan curves in the plane that admit a quasisymmetric parametrization by the unit circle, while Tukia and V\"ais\"al\"a \cite{TukiaVaisala:qs_embeddings} generalized this characterization to arbitrary metric spaces $X$ homeomorphic to the circle. Namely, $X$ admits a quasisymmetric parametrization by the circle if and only if $X$ has bounded turning and is doubling; see Section \ref{section:prelim} for the definitions of these notions.

On the other hand, the uniformization of two-dimensional metric spaces is much more challenging. In this survey we present a series of major results proved during the last three decades on the uniformization of non-smooth metric surfaces, i.e., two-dimensional manifolds with a metric (not necessarily arising from a Riemannian metric) inducing the topology. The presentation is aimed at non-experts and contains all required background, some examples, and proof outlines. Specifically, we will present results that answer the following question.

\begin{question}\label{question:minimal}
What are minimal conditions on a metric surface so that it can be parametrized by a smooth surface with a well-behaved map?
\end{question}

The study of the uniformization problem on metric surfaces is implicitly motivated by work of Thurston. The hyperbolization conjecture of Thurston for $3$-dimensional manifolds asserts that every closed, aspherical, irreducible, atoroidal $3$-manifold $M$ admits a Riemannian metric of constant curvature $-1$. A related conjecture of Cannon \cite{Cannon:Conjecture} asserts that if $G$ is a Gromov hyperbolic group whose boundary at infinity is a topological $2$-sphere, then $G$ admits a discrete, cocompact, isometric action on hyperbolic $3$-space. Cannon's conjecture implies Thurston's conjecture when the fundamental group of $M$ is Gromov hyperbolic. Although Thurston's conjecture has been resolved affirmatively by Perelman, Cannon's conjecture remains open. By results of Sullivan \cite{Sullivan:Cannon} and Tukia \cite{Tukia:qc_groups}, Cannon's conjecture is equivalent to the following conjecture.

\begin{conjecture}
If $G$ is a Gromov hyperbolic group whose boundary at infinity $\partial_\infty G$ is homeomorphic to the $2$-sphere $\mathbb S^2$, then $\partial_\infty G$ is quasisymmetrically equivalent to $\mathbb S^2$.
\end{conjecture}

Therefore, a deep open problem in geometric group theory is closely tied with the problem of quasisymmetric uniformization of metric spheres. 
We direct the reader to the ICM surveys of Bonk \cite{Bonk:icm} and Kleiner \cite{Kleiner:icm} for further background on the conjecture. 

The problem of uniformization of metric spheres arises as well in the context of dynamical systems. Even in that context, the problem is motivated again by work of Thurston. Bonk and Meyer \cite{BonkMeyer:Thurston} studied in depth expanding Thurston maps on the $2$-sphere. These maps give rise to a type of fractal geometry on the sphere, represented by an associated visual metric. One of the main results of the theory, which also appeared in the work of Ha\"issinsky--Pilgrim \cite{HaissinskyPilgrim:coarse}, is that an expanding Thurston map is topologically conjugate to a rational map on the Riemann sphere if and only if the associated visual metric is quasisymmetrically equivalent to the Euclidean metric on $\mathbb S^2$.

\subsection{Organization}
We start in Section \ref{section:classical} with the classical uniformization theorem for Riemann surfaces. In Section \ref{section:prelim} we discuss the required preliminaries from the area of analysis on metric spaces. 

Section \ref{section:qs} contains an exposition of the quasisymmetric uniformization problem and is focused mainly on the Bonk--Kleiner theorem \cite{BonkKleiner:quasisphere}. This is the first modern uniformization result for non-smooth surfaces and ignited a whole line of research that is presented in this survey. We also include a brief outline of the proof of the theorem, which, quite remarkably, relies on the Koebe--Andreev--Thurston circle packing theorem. In the same section we present an improvement of the Bonk--Kleiner theorem that is due to Lytchak and Wenger \cite{LytchakWenger:parametrizations} and provides quasisymmetric parametrizations that are, in addition, canonical, i.e., unique in a sense. 

Section \ref{section:qc} is devoted to the quasiconformal uniformization theorem of Rajala \cite{Rajala:uniformization}. We include a proof outline, motivating examples, and a detailed discussion on the relation between quasiconformal and quasisymmetric maps. In particular, we explain how Rajala's theorem implies the Bonk--Kleiner theorem.

In the final part of the survey, Section \ref{section:wqc}, we present a theorem of the author and Romney \cite{NtalampekosRomney:nonlength} on the weakly quasiconformal uniformization of arbitrary metric surfaces of locally finite area, under no other geometric assumption. A predecessor of this result was obtained by Meier and Wenger \cite{MeierWenger:uniformization} for locally geodesic surfaces. These can be regarded as the ultimate results in the above series of works and they imply the two mentioned uniformization theorems. We include a proof outline of the main theorems of \cite{NtalampekosRomney:nonlength} and we discuss several properties of weakly quasiconformal maps.

\subsection*{Acknowledgments}
The author would like to thank Athanase Papadopoulos for the invitation to write this survey, as well as Toni Ikonen, Damaris Meier, Kai Rajala, and Matthew Romney for their comments and corrections.

\section{Classical conformal uniformization}\label{section:classical}

\subsection{The uniformization theorem}
A Riemann surface is a $2$-dimensional topological manifold $X$ (possibly with non-empty boundary) with a complex structure, that is, an atlas whose transition maps are conformal. We note that Riemann surfaces are orientable. A homeomorphism between Riemann surfaces is conformal if it is complex differentiable in local coordinates. The classical Riemann mapping theorem asserts that every simply connected domain in the complex plane $\C$, other than $\C$ itself, admits a conformal parametrization by the unit disk $\D$. The generalization of this theorem to arbitrary Riemann surfaces, as conjectured by Klein in 1882, is now known as the classical uniformization theorem and is due to Koebe and Poincar\'e \cites{Koebe:uniformization1, Koebe:uniformization2, Koebe:uniformization3, Poincare:uniformization}. We direct the reader to \cite{SaintGervais:uniformization} for the history of the uniformization theorem and to
\cite{Marshall:complex}*{Theorem 15.12} for a proof.

\begin{theorem}[Classical uniformization]\label{theorem:uniformization_classical}
Let $X$ be a simply connected Riemann surface without boundary. If $X$ is compact, then there exists a conformal homeomorphism from $\widehat{\C}$ onto $X$ and if $X$ is not compact, then there exists a conformal homeomorphism from either $\D$ or $\C$ onto $X$. 
\end{theorem}

%**To explain the figure. Assume that $X$ is a smooth submanifold of $\R^3$ and endow $X$ with the induced Riemannian metric. At each point of $X$ we assign an infinitesimal circle. In the local coordinates each such circle corresponds to an infinitesimal ellipse. In the isothermal coordinate each ellipse corresponds to a circle. By uniformization, $f$ is conformal between the plane and the isothermal coordinate so infinitesimal circles in $\C$ are mapped to infinitesimal circles in the isothermal coordinate. Thus infinitesimal circles (or equivalently squares) in $\C$ are mapped to infinitesimal circles on $X$.

Note that a Riemann surface does not carry a natural metric. Suppose that a Riemann surface $X$ is equipped with a Riemannian metric $g$. We say that $g$ is \textit{compatible} with the conformal structure of $X$ if for each local conformal parametrization $\varphi$ from an open set $U$ of $\C$ onto an open subset of $X$ the coordinate representation of $g$ is of the form $\varphi^*g=\alpha |dz|^2$ for some smooth function $\alpha$ on $U$. 

Geometrically one may think of a Riemannian metric $g$ on a Riemann surface $X$ as an assignment of an ellipse to each point of the surface $X$ via local coordinates. With this point of view, $g$ is compatible with the complex structure if and only if $g$ assigns an actual circle to each point rather than an ellipse. In Section \ref{section:complex_structures} we give another interpretation using the notion of quasiconformal maps. An implication of the classical uniformization theorem is the next statement.

%We say that a metric $d$ on a Riemann surface $X$ is \textit{compatible} with the complex structure of $X$ if each local conformal chart $\varphi$ from an open subset $U$ of $X$ into the plane $\C$ is $1$-quasiconformal, where $U$ is equipped with the restriction of the metric $d$. When the metric $d$ arises from a Riemannian metric $g$, we will simply say that $g$ is compatible with the complex structure. An implication of the classical uniformization theorem is the next statement. 

\begin{theorem}[Classical Riemannian uniformization]\label{theorem:uniformization_riemannian}
Let $X$ be a Riemann surface. Then there exists a Riemannian metric $g$ on $X$ that is complete, has constant curvature, and is compatible with the complex structure of $X$. 
\end{theorem}

Conversely, Gauss proved that every orientable Riemannian surface admits a complex structure that is compatible with the Riemannian metric. This complex structure is also referred to as \textit{isothermal coordinates}.

\begin{theorem}[Existence of isothermal coordinates]\label{theorem:isothermal}
Let $(X,g)$ be {an orientable} Riemannian surface. Then there exists a complex structure on $X$ that is compatible with the Riemannian metric $g$.
\end{theorem}

\begin{figure}
	\begin{tikzpicture}
	\clip (-2.3,-1.7) rectangle (8.3,1.8);
	\node at (0,0) {\includegraphics[scale=.37]{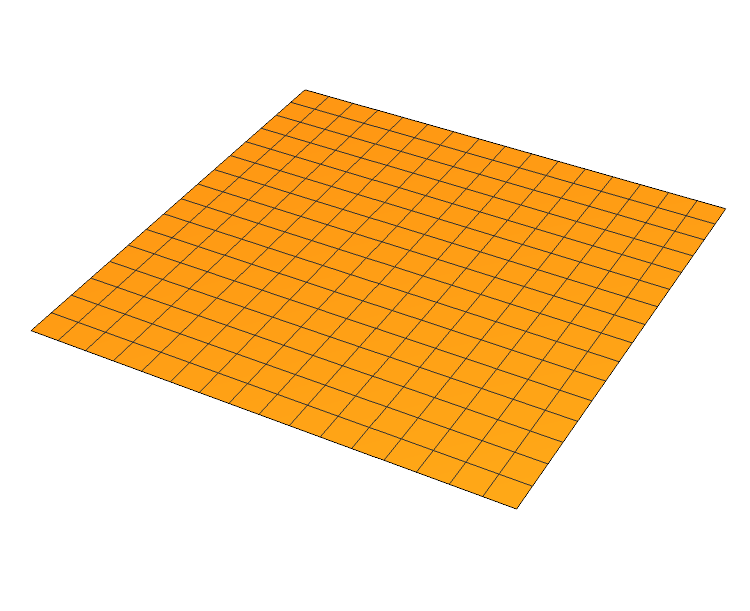}};
	\node at (6,0) {\includegraphics[scale=.37]{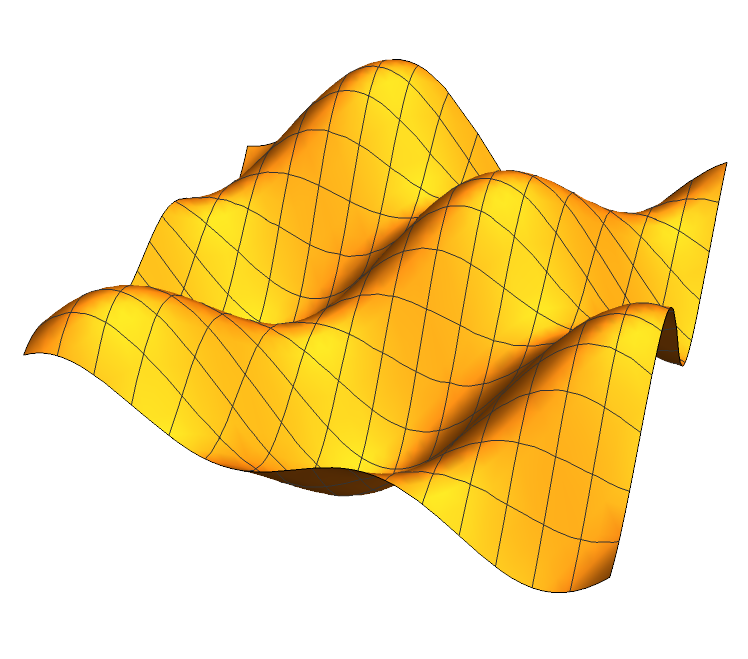}};
	\draw[->] (2.3,0) to[out=30, in=150] (3.5,0);
	\node at (3,0.5) {$f$};
	\end{tikzpicture}
	\caption{Illustration of a conformal map from the plane onto a Riemannian surface embedded in Euclidean space.}\label{figure:classical}
\end{figure}
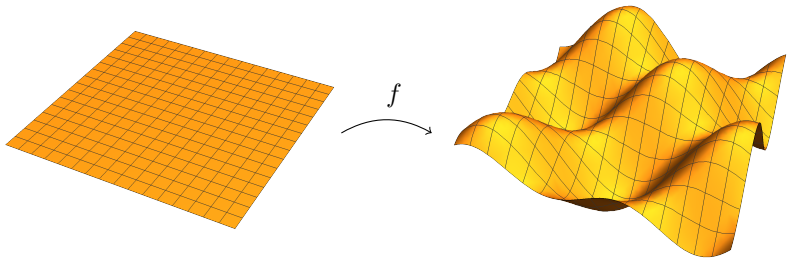

An implication of Theorem \ref{theorem:uniformization_classical} and Theorem \ref{theorem:isothermal} is the following statement.

\begin{theorem}\label{theorem:uniformization_sphere}
Let $(X,g)$ be a Riemannian surface homeomorphic to the $2$-sphere. Then there exists a complex structure on $X$ that is compatible with $g$ and a conformal map from the Riemann sphere $\widehat \C$ onto $X$.
\end{theorem} 

\subsection{Riemann surfaces with a metric structure}\label{section:riemann}

As we discussed, a Riemann surface does not carry a natural metric. However, surfaces embedded in Euclidean space do carry a natural Riemannian metric inherited from the ambient space, which gives rise to a compatible conformal structure. Hence, if $X$ is an embedded smooth submanifold of Euclidean space, then Theorem \ref{theorem:uniformization_sphere} provides us with a conformal map from the Riemann sphere onto $X$. See Figure \ref{figure:classical} for an illustration, where infinitesimal squares (or circles) are mapped to infinitesimal squares (or circles); this actually gives a geometric interpretation of conformality for maps between surfaces embedded in Euclidean space. 

There are some other classes of Riemann surfaces such as polyhedral and Aleksandrov surfaces that carry a natural metric. This metric does not arise from a Riemannian metric, yet it is compatible in a sense with the complex structure; see Section \ref{section:complex_structures}.  

A \textit{polyhedral surface} is a surface constructed by gluing planar polygons along their boundaries. More formally, a polyhedral surface is a locally finite polyhedral complex (as defined in \cite{BridsonHaefliger:metric}*{Definition I.7.37}) that is homeomorphic to a $2$-dimensional manifold. A polyhedral surface carries a natural \textit{length metric} (i.e., the distance between points is equal to the infimum of the lengths of paths connecting these points) that is locally isometric in the interior of each polygon to the Euclidean metric. Note that this metric does not arise from a Riemannian metric because of singularities at the vertices. Moreover, each polyhedral surface $X$ carries a natural complex structure whose atlas contains the maps that identify the polygons of $X$ with the corresponding polygons in $\C$. See \cite{Courant:Dirichlet}*{Section II.4} and \cite{NtalampekosRomney:length}*{Section 2.5.1} for further details. Thus, the classical uniformization theorem applies to polyhedral surfaces; see Figure \ref{figure:polyhedral}.

\begin{figure}
	\begin{tikzpicture}
		\node  at(1,0) {\includegraphics[scale=0.25]{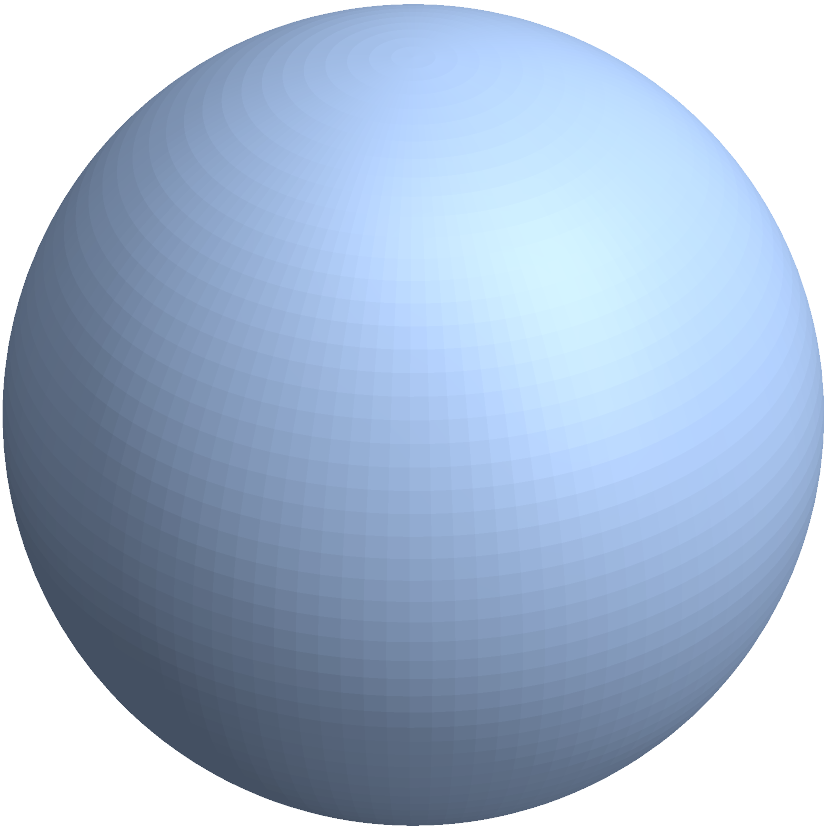}};
		\node  at (7,0) {\includegraphics[scale=0.25]{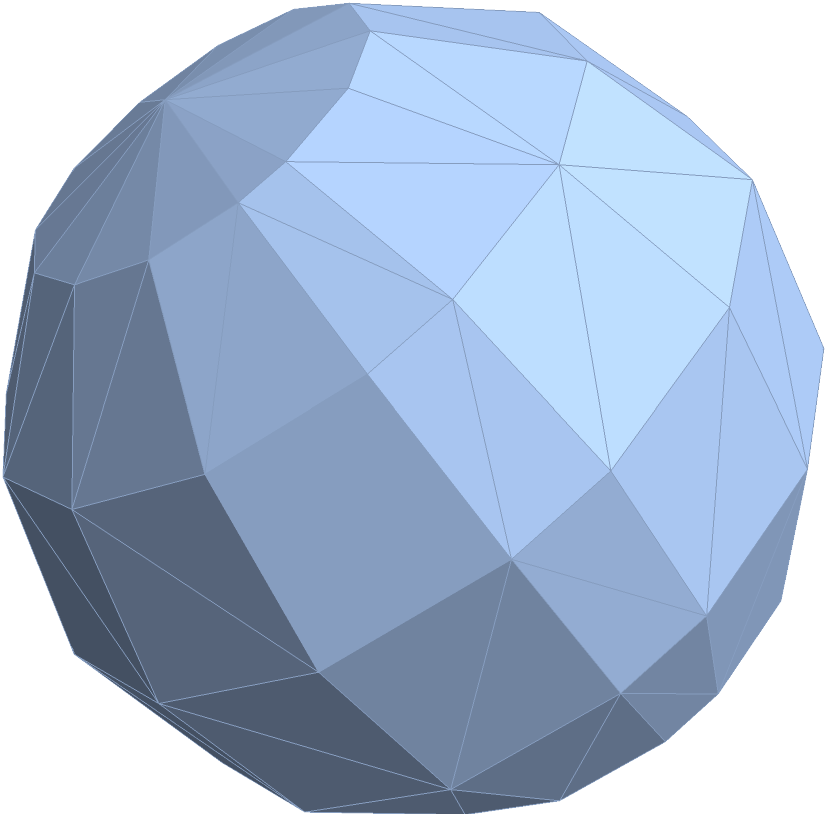}};
		\draw[->] (3,0) to[out=30, in=150](5,0);
		\node at (4,0.6) {$f$};
	\end{tikzpicture}
	\caption{Illustration of a conformal map from the standard sphere onto a polyhedral sphere.}\label{figure:polyhedral}
\end{figure}

An \textit{Aleksandrov surface}, according to the analytic definition, is a $2$-manifold with a length metric that admits a complex structure such that in local coordinates the metric can be defined by a length element of the form
$$e^{u(z)}|dz|,$$
where $u$ is the difference of two subharmonic functions such that $e^u$ is locally integrable on analytic curves. Aleksandrov surfaces contain the classes of Riemannian and polyhedral surfaces and, roughly speaking, they form the largest class of surfaces on which a notion of integral curvature can be defined. For that reason, initially those surfaces were termed as \textit{surfaces of bounded curvature}. It is known that Aleksandrov surfaces are exactly those surfaces that can be approximated by polyhedral or Riemannian surfaces satisfying local integral curvature bounds. We direct the reader to \cites{AleksandrovZalgaller:intrinsic, Huber:aleksandrov, Reshetnyak:bounded_curvature} for details and also to \cites{BonkLang:biLipschitz, Troyanov:aleksandrov} for an overview of Aleksandrov surfaces.

We will return again to Riemannian, polyhedral, and Aleksandrov surfaces in Section \ref{section:complex_structures} and we will discuss in more detail the compatibility of their metric and the complex structure. More generally, we will present results in response to the following question.

\begin{question}\label{question:structure}
What is the largest class of metric surfaces that admit a conformal structure compatible with the metric?
\end{question}

\section{Preliminaries on metric spaces}\label{section:prelim}

\subsection{Geometric notions}\label{section:metric}

Each Riemannian manifold $(X,g)$ carries an intrinsic metric $d_g$ induced by $g$. To simplify notation, we will implicitly assume that a Riemannian manifold is equipped with this metric and we will denote the resulting metric space by $(X,g)$ rather than $(X,d_g)$. A \textit{metric surface} is a metric space $(X,d)$ that is homeomorphic to a topological $2$-manifold $X$ with or without boundary.

Let $(X,d)$ be a metric space. We denote by $B(x,r)$ the open ball centered at a point $x\in X$ and with radius $r>0$. In order to signify the metric $d$, if necessary, we use instead the notation $B_d(x,r)$. The same convention applies to other metric notions. For example, the diameter of a set $E\subset X$ is denoted by $\diam(E)$ or equivalently by $\diam_d(E)$.

For $s>0$ the \textit{Hausdorff $s$-measure} of a set $E\subset X$ is defined as 
$$\mathcal H^s(E)= \lim_{\delta\to 0^+}\mathcal H^s_\delta(E),$$
where 
$$\mathcal H^s_\delta(E)= \frac{2^{-s}\pi^{s/2}}{\Gamma(2^{-1}s+1)}\inf\left\{\sum_{i\in I} \diam(E_i)^s: E\subset \bigcup_{i\in I}E_i,\,\, \diam(E_i)<\delta,\, i\in I \right\}$$
for $\delta\in (0,\infty]$. Here $\Gamma$ denotes the usual gamma function. We will be interested mostly in the case that $s=2$. Then the normalizing constant in the definition of $\mathcal H^s_\delta$ takes the value $\pi/4$ and $\mathcal H^2$ is also called the area measure on $X$. With this normalization, Hausdorff $2$-measure in the plane simply agrees with Lebesgue measure. The Lebesgue measure of a set $E\subset \C$ is alternatively denoted by $|E|$.

We say that the space $X$ is \textit{Ahlfors $s$-regular} if there exists a constant $M\geq 1$ such that 
$$M^{-1}r^s\leq \mathcal H^s(B(x,r))\leq Mr^s$$
for each $x\in X$ and $0<r<\diam (X)$.  We say that $X$ is \textit{linearly locally connected} or $\llc$ in short if there exists a constant $M\geq 1$ such that for each ball $B(a,r)$ in $X$ the following two conditions hold.
\begin{enumerate}[label=\normalfont($\llc_{\arabic*}$)]
\item\label{llc1} For every $x,y\in B(a,r)$ there exists a continuum $E\subset B(a,Mr)$ that contains $x$ and $y$.
\item\label{llc2} For every $x,y\in X\setminus B(a,r)$ there exists a continuum $E\subset X\setminus B(a,r/M)$ that contains $x$ and $y$.
\end{enumerate}
In that case we say that $X$ is $M$-$\llc$. Geometrically, on a surface the $\llc$ condition prevents in a quantitative manner cusps, thin bottlenecks, and dense wrinkles; see Figure \ref{figure:llc} for an illustration. Condition \ref{llc1} is also called the \textit{bounded turning} condition.

We say that $X$ is a \textit{doubling space} if there exists a constant $M\geq 1$ such that for every $R>0$, every ball of radius $R$ can be covered by at most $M$ balls of radius $R/2$. In that case we say that $X$ is $M$-doubling. Note that if $X$ is Ahlfors $s$-regular for some $s>0$, then it is also doubling. Geometrically, the doubling condition expresses in a quantitative way the quality of a space being finite dimensional. For example, $n$-dimensional Euclidean space $\R^n$ is $M(n)$-doubling for a constant $M(n)$ that converges to $\infty$ as $n\to\infty$. On the other hand, $\ell^\infty(\N)$ is not doubling. 

\begin{figure}
	\includegraphics{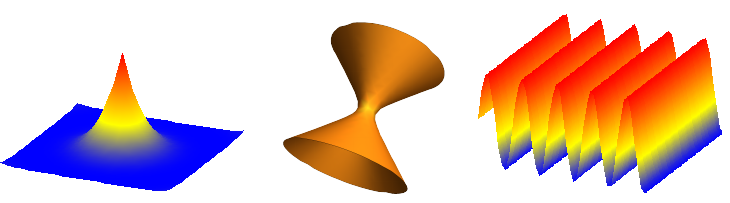}
	\caption{Surfaces with a cusp, a thin bottleneck, and dense wrinkles, respectively.}\label{figure:llc}
\end{figure}

Let $(X,d,\mu)$ be a metric measure space. Here $\mu$ is assumed to be an outer measure that is Borel regular. For a set $G\subset X$ and disjoint sets $E,F\subset G$ we define $\Gamma(E,F;G)$ to be the family of curves in $G$ joining $E$ and $F$. Let $\Gamma$ be a family of curves in $X$. A Borel function $\rho\colon X\to [0,\infty]$ is \textit{admissible} for $\Gamma$ if $\int_\gamma \rho\, ds\geq 1$ for all rectifiable paths $\gamma\in \Gamma$. Let $Q\geq 1$. We define the \textit{$Q$-modulus} of $\Gamma$ as
$$\Mod_Q \Gamma=\inf_\rho \int_X \rho^Q \, d\mu,$$
where the infimum is taken over all admissible functions $\rho$ for $\Gamma$. When $X$ is a metric surface of locally finite Hausdorff $2$-measure, $\mu=\mathcal H^2$, and $Q=2$ we will refer to $\Mod_2$ as the conformal modulus on $X$. Modulus is an outer measure on families of curves in $X$, so it conveys the size of a curve family; see Figure \ref{fig:modulus_plane} for an illustration.

We say that a metric measure space $(X,d,\mu)$ is a \textit{$Q$-Loewner space} if there exists a decreasing function $\varphi\colon (0,\infty)\to (0,\infty)$ such that for every pair of disjoint non-degenerate continua $E,F\subset X$ we have
$$\Mod_Q\Gamma(E,F;X)\geq \varphi(\Delta(E,F)),$$
where $\Delta(E,F)$ is the \textit{relative distance} of $E$ and $F$, which is defined by 
$$ \frac{\dist(E,F)}{\min\{\diam (E),\diam(F)\}}.$$
Geometrically, in a Loewner space every two continua that are large and close to each other can be connected by a rich family of curves. We direct the reader to \cite{Heinonen:metric} for further background on the above metric notions.

\subsection{Metric Sobolev spaces}
Let $(X,d_X,\mu)$ be a metric measure space such that $(X,d_X)$ is separable and $\mu$ is a locally finite Borel regular outer measure on $X$. Let $(Y,d_Y)$ be another metric space and consider a map $h\colon (X,d_X)\to (Y,d_Y)$. We say that a Borel function $g\colon X\to [0,\infty]$ is an \textit{upper gradient} of $h$ if 
\begin{align}\label{ineq:upper_gradient}
    d_Y(h(x),h(y)) \leq \int_{\gamma} g \, ds
\end{align}
for all $x,y\in X$ and every rectifiable path $\gamma$ in $X$ joining $x$ and $y$. This is called the \textit{upper gradient inequality}. If, instead, the above inequality holds for all curves $\gamma$ outside a curve family of $p$-modulus zero, where $p\geq 1$, then we say that $g$ is a \textit{$p$-weak upper gradient} of $h$. In this case, there exists a curve family $\Gamma_0$ with $\Mod_p \Gamma_0=0$ such that all paths outside $\Gamma_0$ and all subpaths of such paths satisfy the upper gradient inequality. The weak upper gradient plays the role of the absolute value of the gradient of a smooth function defined on an open set in Euclidean space.

Let $L^p(X)$ denote the space of $p$-integrable Borel functions from $X$ to the extended real line $\widehat{\mathbb{R}}$, where two functions are identified if they agree $\mu$-almost everywhere. The Sobolev space $N^{1,p}(X,Y)$ is defined as the space of Borel maps $h \colon X \to Y$ that are essentially separably valued (see \cite{HeinonenKoskelaShanmugalingamTyson:Sobolev}*{Section 3.1}) and have a $p$-weak upper gradient $g\in L^p(X)$ such that the function $x \mapsto d_Y(y,h(x))$ is in $L^p(X)$ for some $y \in Y$. The assumption on essential separable values is automatically true when $Y$ is a separable space, which will be the case in all considerations in the present note. If $Y=\R$, we simply write $N^{1,p}(X)$. Note that when $X$ is an open subset of Euclidean space, then $N^{1,p}(X)$ coincides (after appropriate identifications) with the usual Sobolev space $W^{1,p}(X)$. The spaces $L_{\loc}^p(X)$ and $N_{\loc}^{1,p}(X, Y)$ are defined in the obvious manner. Each map $h\in N_{\loc}^{1,p}(X,Y)$ has a unique \textit{minimal} $p$-weak upper gradient $g_h\in L^p_{\loc}(X)$, in the sense that for any other $p$-weak upper gradient $g\in L^p_{\loc}(X)$ we have $g_h(x)\leq g(x)$ for $\mu$-a.e. $x\in X$. See the monograph \cite{HeinonenKoskelaShanmugalingamTyson:Sobolev} for background on metric Sobolev spaces.

\subsection{Convergence of metric spaces}
Let $(X,d_X)$ be a metric space and let $E\subset X$ and $\varepsilon>0$. We say that $E$ is \textit{$\varepsilon$-dense} (in $X$) if for each $x\in X$ we have $\dist(x,E)<\varepsilon$ or equivalently $N_{\varepsilon}(E)=X$, where $N_\varepsilon(E)$ denotes the open $\varepsilon$-neighborhood of $E$. A map $f \colon (X,d_X) \to (Y,d_Y)$ between metric spaces is an \textit{$\varepsilon$-isometry} if $f(X)$ is $\varepsilon$-dense in $Y$ and for each $x,y \in X$ we have
$$|d_X(x,y) - d_Y(f(x),f(y))| < \varepsilon$$

We define the \textit{Hausdorff distance} of two sets $E,F\subset X$ to be the {infimal value} $r>0$ such that $E\subset N_{r}(F)$ and $F\subset N_{r}(E)$. We denote the Hausdorff distance by $d_H(E,F)$. A sequence of sets $E_n\subset X$, $n\in \N$, \textit{converges in the Hausdorff sense} to a set $E\subset X$ if $d_H(E_n,E)\to 0$ as $n\to\infty$.

The \textit{Gromov--Hausdorff distance} between two metric spaces $X,Y$ is defined as the infimal value $r>0$ such that there is a metric space $Z$ with subsets $\widetilde{X}, \widetilde{Y} \subset Z$ such that $X$ and $Y$ are isometric to $\widetilde{X}$ and $\widetilde{Y}$, respectively, and $d_H(\widetilde{X}, \widetilde{Y}) < r$. This is denoted by $d_{GH}(X,Y)$. We note that if there exists an $\varepsilon$-isometry from $X$ to $Y$ for some $\varepsilon>0$, then $d_{GH}(X,Y)<2\varepsilon$ \cite{BuragoBuragoIvanov:metric}*{Corollary 7.3.28}. We say that a sequence of metric spaces $X_n$, $n\in \N$, \textit{converges in the Gromov--Hausdorff sense} to a metric space $X$ if $d_{GH}(X_n,X) \to 0$ as $n \to \infty$. By \cite{BuragoBuragoIvanov:metric}*{Corollary 7.3.28}, this is equivalent to the requirement that there exists a sequence of $\varepsilon_n$-isometries $f_n\colon X_n\to X$, where $\varepsilon_n>0$ and $\varepsilon_n\to 0$ as $n\to\infty$. In this case, we say that $f_n$, $n\in \N$, is an \textit{approximately isometric sequence}. See \cite{BuragoBuragoIvanov:metric}*{Section 7} and  \cite{Petrunin:metric_geometry}*{Section 5} for more background.

\section{Quasisymmetric uniformization}\label{section:qs}

\subsection{Quasisymmetric maps}
Quasisymmetric maps generalize the notion of conformal maps in Euclidean space and can be defined in arbitrary metric spaces, even without the presence of a smooth structure. A homeomorphism $\eta\colon [0,\infty)\to [0,\infty)$ is called a \textit{distortion function}. Let $(X,d_X)$, $(Y,d_Y)$ be metric spaces and $f\colon X\to Y$ be a homeomorphism. We say that $f$ is \textit{quasisymmetric} if there exists a distortion function $\eta$ such that for every triple of distinct points $x,y,z\in X$ we have
$$\frac{d_Y(f(x),f(y))}{d_Y(f(x),f(z))}\leq \eta\left(\frac{d_X(x,y)}{d_X(x,z)}\right).$$
In that case we say that $f$ is $\eta$-quasisymmetric. From a geometric point of view, quasisymmetric maps  preserve relative sizes. The definition implies that there exists a constant $M\geq 1$ depending only on $\eta$ such that for each $x\in X$ and $r>0$ there exists $R>0$ satisfying 
$$B(f(x),R)\subset f(B(x,r)) \subset B(f(x),MR).$$
See Figure \ref{figure:qs} for an illustration. Moreover, under some conditions on $X$ and $Y$, these inclusions can be taken as the definition of a quasisymmetric map; see \cite{Heinonen:metric}*{Theorem 10.19}.

\begin{figure}
	\begin{tikzpicture}[scale=.4,x=1.0cm,y=1.0cm]

\draw [line width=.5pt,fill=black,fill opacity=0.1] (0.76,-0.2) circle (2.984cm);

\draw [fill=black] (0.76,-0.2) circle (1.5pt);
\draw[color=black] (0.9,0.2) node {$\scriptstyle x$};
\draw(4,-3) node {$\scriptstyle B(x,r)$};

\draw (6,1.5) node {$f$};
\draw[->] (5,0.4) to [out=30, in=150] (7,0.5);

\draw[line width=.5pt,fill=black,fill opacity=0.1, rounded corners=7] (10.74,2.4) -- (10.6,1.1) -- (9.04,0.46) --(10,-2)-- (11.58,-2.42) -- (13,-3)--(14.94,-0.9) -- (14.4,1)-- (14.46,2.94) -- (12.78,2.06)  -- cycle;

\draw [fill=black] (12.54,0.18) circle (1.5pt);
\draw[color=black] (12.7,0.7) node {$\scriptstyle f(x)$};
\draw(12,-1.7) node {$\scriptstyle f(B(x,r))$};

\draw [line width=.5pt] (12.54,0.18) circle (1.23cm); 
\draw (12.54,0.18) -- (13.77,0.18) node[pos=0.5,below]{$\scriptstyle R$};

\draw [line width=.5pt] (12.54,0.18) circle (4.36cm);

\draw(18,-3) node {$\scriptstyle B(f(x),MR)$};

\end{tikzpicture}
	\caption{The definition of a quasisymmetric map.}\label{figure:qs}
\end{figure}
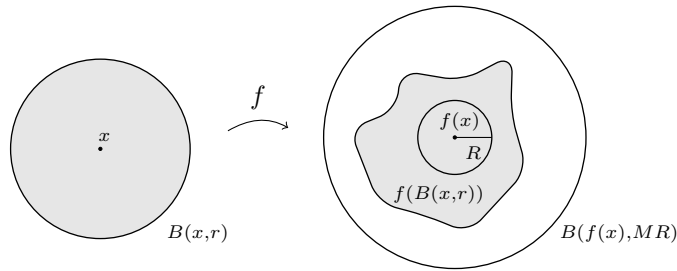

A homeomorphism $f\colon X\to Y$ is a \textit{bi-Lipschitz map} if there exists $L\geq 1$ such that
$$L^{-1}d_X(x,y)\leq  d_Y(f(x),f(y))\leq Ld_X(x,y)$$
for all $x,y\in X$. Note that every bi-Lipschitz map is quasisymmetric, but the converse is not true.

\subsection{The Bonk--Kleiner theorem}

The first milestone on the uniformization of non-smooth metric surfaces is due to Bonk and Kleiner \cite{BonkKleiner:quasisphere}, who provided a satisfactory set of sufficient conditions for the quasisymmetric uniformization of metric $2$-spheres.

It is elementary to show based on the definition that the $\llc$ condition is a quasisymmetric invariant. This was first observed by V\"ais\"al\"a \cite{Vaisala:quasimobius}*{Theorems 4.4 and 4.5}; see also \cite{Rehmert:thesis}*{Proposition 2.1.7} for a short argument. Note that every compact Riemannian manifold is $\llc$. Therefore, if $X$ is a metric $2$-sphere, the $\llc$ property is a necessary condition for the existence of a quasisymmetric parametrization of $X$ by the Riemann sphere $\widehat \C$, which is equipped with the spherical metric $\sigma$. Bonk and Kleiner proved the striking result that for Ahlfors $2$-regular spheres the $\llc$ condition is also sufficient for quasisymmetric uniformization.

\begin{theorem}[Bonk--Kleiner theorem]\label{theorem:bonk_kleiner}
Let $X$ be a metric $2$-sphere that is Ahlfors $2$-regular. Then there exists a quasisymmetric homeomorphism $f\colon \widehat \C\to X$ if and only if $X$ is $\llc$. The statement is quantitative. 
\end{theorem}

A statement is said to be \textit{quantitative} if the parameters or constants in the conclusions depend only on the parameters or constants in the assumptions. 

We remark that one cannot expect the parametrization $f$ to be bi-Lipschitz, as was shown by Laakso \cite{Laakso:plane}. Moreover, higher dimensional analogues of the Bonk--Kleiner theorem are known to be false. Specifically, Semmes \cite{Semmes:good} proved that there exist $\llc$ (in fact, linearly locally contractible) and Ahlfors $3$-regular metric spaces homeomorphic to the $3$-sphere $\mathbb S^3$, but not quasisymmetrically equivalent to $\mathbb S^3$. Examples in all dimensions $n\geq 3$ were constructed in \cites{HeinonenWu:whitehead, PankkaWu:geometry_qs}.

\begin{proof}[Outline of proof]
In the non-trivial direction, one needs to find a parametrization of $X$ with nice geometric properties. If $X$ were a smooth manifold, the classical uniformization theorem would provide such a map. In the absence of a smooth structure the proof of Bonk and Kleiner relies on the Koebe--Andreev--Thurston circle packing theorem and on modern techniques from the area of analysis on metric spaces in order to produce a parametrization. 

The Ahlfors $2$-regularity assumption implies that $X$ is a doubling space. The first step is to use the $\llc$ and doubling properties to produce a good graph approximation of $X$ that is the $1$-skeleton of a triangulation of $X$. In the second step, the Koebe--Andreev--Thurston circle packing theorem then provides a circle packing in the Euclidean $2$-sphere $\mathbb S^2$ (or equivalently the Riemann sphere $\widehat \C$) and a map $f$ from the vertex set of the triangulation of $X$ to the centers of disks of the circle packing of $\mathbb S^2$. 

The idea is to show that this map converges to a quasisymmetric map, upon considering finer and finer triangulations of $X$. To this end, and as the third and final step, one shows that after suitable normalizations the map $f$ is $\eta$-quasisymmetric for some uniform distortion function $\eta$ that depends only on the data, i.e., the Ahlfors $2$-regularity and $\llc$ constants.

This final step is achieved with the aid of modulus estimates. The map $f^{-1}$ preserves a type of discrete modulus defined on the tangency graph of the circle packing of $\mathbb S^2$ and on the $1$-skeleton of the triangulation of $X$. The geometric assumptions on $X$ allow for conformal modulus bounds on $X$, as well as comparison principles between conformal and discrete modulus. The punchline of the argument is given by (a discrete analogue of) the principle that a modulus-preserving map between a  $2$-Loewner, doubling space and an $\llc$, Ahlfors $2$-regular space is quasisymmetric (see Theorem \ref{theorem:qc_qs} below).
\end{proof}

Since the original proof, several other proofs of comparable difficulty and novelty have been published by Rajala, Lytchak--Wenger, Meier--Wenger, and Romney jointly with the author \cites{Rajala:uniformization, LytchakWenger:parametrizations, MeierWenger:uniformization, NtalampekosRomney:length, NtalampekosRomney:nonlength}. We will discuss some of those results later. Remarkably, the result of Lytchak and Wenger provides canonical parametrizations of metric surfaces under the assumptions of Bonk--Kleiner. On the other hand, the parametrization given by the proof of Bonk and Kleiner is by no means canonical. 

The Bonk--Kleiner theorem has been used in the problem of quasisymmetric embedding of Sierpi\'nski carpets into the standard sphere \cites{Haissinsky:hyperbolic_groups,  CheegerErikssonBique:thin, HakobyanLi:qs_embeddings},  
in the construction of quasispheres \cites{VellisWu:qs, Vellis:chordarc}, as well as in connection with the problem of conformal removability \cite{Ntalampekos:gasket}.

\subsection{Canonical quasisymmetric parametrization}
A major contribution in the problem of quasisymmetric parametrization was made by Lytchak and Wenger \cite{LytchakWenger:parametrizations}, who proved that Ahlfors $2$-regular and $\llc$ spheres admit canonical quasisymmetric parametrizations. Their proof exploits a series of results that they developed while studying minimal surfaces in metric spaces. 
Let $(X,d_X)$ and $(Y,d_Y)$ be metric surfaces equipped with the Hausdorff $2$-measure and let $u\in N^{1,2}(X,Y)$ be a map with minimal $2$-weak upper gradient $g_u\in L^2(X)$. The \textit{Reshetnyak energy} of $u$ is defined by 
$$E^2_+(u)= \int_X g_u^2\, d\mathcal H^2.$$

\begin{theorem}[\cite{LytchakWenger:parametrizations}*{Theorem 6.2}]\label{theorem:canonical}
Let $X$ be a metric $2$-sphere that is Ahlfors $2$-regular and $\llc$. Then among all maps $v\in N^{1,2}(\widehat \C, X)$ that are uniform limits of homeomorphisms there exists a map $u\colon \widehat \C\to X$ with minimal Reshetnyak energy. The map $u$ is a quasisymmetric homeomorphism and is uniquely determined up to a M\"obius transformation of $\widehat \C$. 
\end{theorem}

We give a proof outline to describe the main tools appearing in the argument. We refer the reader to \cite{LytchakWenger:parametrizations} for the definitions of the various notions appearing in the outline.

\begin{proof}[Outline of proof]
The first step is to show that a metric $2$-sphere $X$ that is Ahlfors $2$-regular and $\llc$ admits a \textit{quadratic isoperimetric inequality} \cite{LytchakWenger:parametrizations}*{Theorem 1.4}. That is, every Lipschitz curve $\gamma\colon \mathbb S^1\to X$ is the trace of some function $u\in N^{1,2}(\D,X)$ such that the \textit{parametrized area} of $u$ is bounded above by $C\ell(\gamma)^2$ for some uniform constant $C>0$. 

Given a Jordan curve $\gamma$ in $X$, we denote by $\Lambda(\gamma,X)$ the possibly empty family of maps $v\in N^{1,2}(\D, X)$ whose trace has a continuous representative that \textit{weakly monotonically} parametrizes $\gamma$. Note that if $\gamma$ has finite length, then it admits an injective Lipschitz parametrization, so $\Lambda(\gamma,X)\neq \emptyset$ under the quadratic isoperimetric inequality. Lytchak and Wenger proved in earlier works \cites{LytchakWenger:regularity, LytchakWenger:discs, LytchakWenger:minimizers} that whenever $\Lambda(\gamma,X)$ is non-empty, and under the quadratic isoperimetric inequality, there exists a function $u\in N^{1,2}(\D,X)$ of minimal Reshetnyak energy such that $u$ has a continuous extension to $\br \D$ and is \textit{infinitesimally isotropic} (roughly speaking, this means that it maps infinitesimal balls to infinitesimal balls).

It is then shown that $u$ must be a uniform limit of homeomorphisms \cite{LytchakWenger:parametrizations}*{Theorem 1.2}. This property, together with infinitesimal isotropy, imply that
\begin{align}\label{ineq:canonical}
\Mod_2\Gamma \leq \frac{4}{\pi}\Mod_2 u(\Gamma)
\end{align}
for every curve family $\Gamma$ in $\br \D$. This inequality and the Ahlfors $2$-regularity assumption force $u$ to be a homeomorphism \cite{LytchakWenger:parametrizations}*{Theorem 3.6}. Finally, a standard argument (see Theorem \ref{theorem:qc_qs} below) based on the modulus inequality \eqref{ineq:canonical} and the geometric assumptions on $X$ shows that $u$ is quasisymmetric. The uniqueness of $u$ up to M\"obius transformations follows from the energy minimizing properties and specifically from the fact that $u$ is {infinitesimally isotropic}.

The above outlined argument allows the canonical parametrization of Jordan regions in $X$. One can modify the process to find a canonical parametrization of the entire sphere $X$; see \cite{LytchakWenger:parametrizations}*{p.~791}.
\end{proof}

A fairly detailed picture of the uniformization of surfaces admitting a quadratic isoperimetric inequality was given by Creutz and Romney \cite{CreutzRomney:branch}, with focus on the question whether a space admits a canonical parametrization. A metric space is called a \textit{chord-arc curve} if it bi-Lipschitz equivalent to the unit circle $\mathbb S^1$. 

\begin{theorem}[\cite{CreutzRomney:branch}*{Theorem 1.2}]
Let $X$ be a geodesic metric space homeomorphic to $\br \D$ that has finite Hausdorff $2$-measure and admits a quadratic isoperimetric inequality. Suppose that $\partial X$ is a chord-arc curve. The following are equivalent.
\begin{enumerate}[label=\normalfont(\arabic*)]
	\item $X$ is quasisymmetrically equivalent to $\br \D$.
	\item $X$ is doubling.
	\item $X$ is Ahlfors $2$-regular.
\end{enumerate}
\end{theorem}

Theorem \ref{theorem:canonical} was generalized to compact surfaces by Fitzi and Meier \cite{FitziMeier:canonical}. 
%For a Riemannian $2$-manifold $(M,g)$ and a metric space $X$ denote by $\Lambda((M,g),X)$ the possibly empty family of maps $v\in N^{1,2}((M,g),X)$ such that $v$ is a uniform limit of homeomorphisms. Also, denote by $E^2_+(u,g)$ the Reshetnyak energy of a map $u\in N^{1,2}((M,g),X)$. 

\begin{theorem}[\cite{FitziMeier:canonical}*{Theorem 1.1}]\label{theorem:fitzimeier}
Let $M$ be a compact smooth $2$-manifold (with boundary) and $X$ be a metric space homeomorphic to $M$ that is Ahlfors $2$-regular and $\llc$. Then among all Riemannian metrics $h$ on $M$ and all maps $v\in N^{1,2}((M,h), X)$ that are uniform limits of homeomorphisms there exists a Riemannian metric $g$ on $M$ and a map $u\colon (M,g)\to X$ with minimal Reshetnyak energy. The map $u$ is a quasisymmetric homeomorphism and the pair $(u,g)$ is uniquely determined up to a conformal diffeomorphism $\varphi\colon (M,g)\to (M,h)$.
\end{theorem}

Generalizations of the Bonk--Kleiner theorem to compact surfaces other than the sphere have been established prior to the above result by Geyer--Wildrick \cite{GeyerWildrick:uniformization} and Ikonen \cite{Ikonen:isothermal}, who obtain the existence of (not necessarily canonical) quasisymmetric parametrizations of closed (i.e., compact, without boundary) surfaces as in Theorem \ref{theorem:fitzimeier} by smooth surfaces. The advantage of \cites{GeyerWildrick:uniformization, Ikonen:isothermal} is that the results are quantitative, as opposed to Theorem \ref{theorem:fitzimeier}, which is qualitative.

\subsection{Quasispheres of finite area}

We return to the case of spheres for the sake of simplicity. A metric space $X$ is a ($2$-dimensional) \textit{quasisphere} if there exists a quasisymmetric homeomorphism $f\colon \widehat \C\to X$. Although the $\llc$ condition is necessary for a quasisphere, Ahlfors $2$-regularity is too restrictive. In fact, quasispheres can even have infinite area. In the case of quasispheres of finite area (with some additional assumptions) one may obtain some characterizations in terms of the metric notions defined in Section \ref{section:metric}. In the case of smooth spheres, we have the next result, which is proved via the classical uniformization theorem; see  \cite{Ntalampekos:qs_approximation}*{Theorem 1.1}. A generalization of this result to the non-smooth case appears in Theorem \ref{theorem:reciprocal_qs}.

\begin{theorem}[Smooth quasispheres]\label{theorem:smooth}
Let $X$ be a Riemannian $2$-sphere. The following statements are quantitatively equivalent.
\begin{enumerate}[label=\normalfont{(\Alph*)}]
\item\label{smooth:quasisphere} $X$ is a quasisphere.
\item\label{smooth:llc-mod}
\begin{enumerate}[label=\normalfont{(\arabic*)}]
\item\label{smooth:llc} $X$ is a doubling and linearly locally connected space and
\item\label{smooth:modulus} there exist constants $L>1$ and $M>0$ such that for every ball $B(a,r)\subset X$ we have 
$$\Mod_2 \Gamma(\br B(a,r), X\setminus B(a,L r);X) <M.$$
\end{enumerate} 
\item\label{smooth:loewner} $X$ is a doubling and $2$-Loewner space. 
\end{enumerate}
\end{theorem}

Of course a smooth sphere satisfies \ref{smooth:llc-mod} and \ref{smooth:loewner} trivially, since it is locally bi-Lipschitz to the Euclidean plane. So the content of the theorem is about the quantitative relation of the quasisymmetric distortion function with the various parameters in \ref{smooth:llc-mod} and \ref{smooth:loewner}. Moreover, condition \ref{smooth:llc-mod}\ref{smooth:modulus} is a consequence of Ahlfors $2$-regularity \cite{Heinonen:metric}*{Lemma 7.18}. We remark that Semmes \cite{Semmes:chordarc2}*{Theorem 5.4} had already provided sufficient conditions in the spirit of Bonk--Kleiner for a smooth $2$-sphere to be a quasisphere, quantitatively. The next result concerns arbitrary non-smooth spheres of finite area.

\begin{theorem}[\cite{Ntalampekos:qs_approximation}*{Theorem 1.3}]\label{theorem:bac}
Let $X$ be a metric $2$-sphere of finite Hausdorff $2$-measure. Then the implications \ref{smooth:llc-mod} $\Rightarrow$ \ref{smooth:quasisphere} $\Rightarrow$ \ref{smooth:loewner} in Theorem \ref{theorem:smooth} are true, quantitatively.
\end{theorem}

In particular, the implication \ref{smooth:llc-mod} $\Rightarrow$ \ref{smooth:quasisphere} implies the Bonk--Kleiner theorem. The proof relies on the deep uniformization result for metric surfaces of locally finite area \cite{NtalampekosRomney:nonlength} that we present in Section \ref{section:wqc}. We discuss the reverse implications in Section \ref{section:reciprocal}. The implication  \ref{smooth:quasisphere} $\Rightarrow$ \ref{smooth:llc-mod} does not hold in general, while the implication \ref{smooth:loewner} $\Rightarrow$ \ref{smooth:quasisphere} is an open question. 

\begin{question}
Let $X$ be a metric sphere of finite Hausdorff $2$-measure. If $X$ is a doubling and $2$-Loewner space, is it a quasisphere?
\end{question}

\subsection{Arbitrary quasispheres}
A typical example of a quasisphere that is not Ahl\-fors $2$-regular and has infinite area is the \textit{snowsphere} of Meyer \cite{Meyer:origami}; see Figure \ref{figure:snowsphere}. The snowsphere is constructed as follows. Consider the surface of the unit cube in $\R^3$. Each face is a unit square and we subdivide it into nine squares of side length $1/3$. Then we replace the middle square with a cubical cap, consisting of five faces of side length $1/3$. We repeat this subdivision and replacement process in each square of side length $1/3$, etc. The space that we obtain in each stage of the construction is a polyhedral surface with its intrinsic metric and it is actually a quasisphere with uniform parameters. The snowsphere is the Gromov--Hausdorff limit of that sequence of spaces.

\begin{figure}
\centering
\includegraphics[scale=0.5]{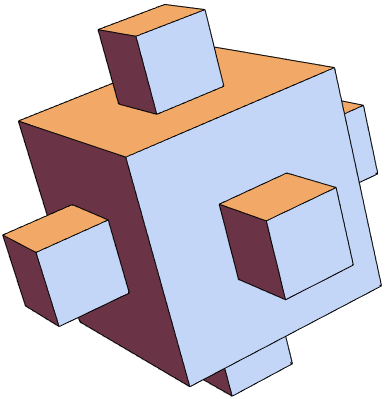}
\caption{The first stage of the construction of the snowsphere.}\label{figure:snowsphere}
\end{figure}

The example of the snowsphere has a direct connection with the dynamics of expanding Thurston maps on the $2$-sphere, the theory of which was thoroughly explored in \cite{BonkMeyer:Thurston}. One of the key results is that for an expanding Thurston map $f$ of the $2$-sphere, the associated visual metric on the $2$-sphere is quasisymmetric to the Euclidean metric if and only if the map $f$ is topologically conjugate to a rational map.

The next theorem proved in \cite{Ntalampekos:qs_approximation} allows us to study arbitrary quasispheres through approximation by smooth quasispheres. The result shows that every quasisphere resembles in a sense the snowsphere, which is constructed as a limit of polyhedral quasispheres.

\begin{theorem}[Smooth approximation of quasisymmetric $2$-manifolds]\label{theorem:main:approximation}
Let $(X,g)$ be a compact Riemannian $2$-manifold (with boundary) and let $d_X$ be a metric on $X$ that induces its topology. The following are quantitatively equivalent.
	\begin{enumerate}[label=\normalfont(\arabic*)]
	\item\label{theorem:main:approximation:1} The metric space $(X,g)$ is quasisymmetric to $(X,d_X)$.
	\item\label{theorem:main:approximation:2}  For each $k\in \N$ there exists a Riemannian metric $g_k$ on $X$ and a metric $d_k$ that is locally isometric to the intrinsic metric $d_{g_k}$ such that 
	\begin{align*}
	&\text{$(X,d_k)$ converges to $(X,d_X)$ in the Gromov--Hausdorff metric as $k\to\infty$ and}\\
	&\text{$(X,d_{k})$ is uniformly quasisymmetric to $(X,g)$ for each $k\in \N$.}
	\end{align*}
	\end{enumerate}
%In this case there exists an approximately isometric sequence $f_k\colon (X,d_k) \to (X,d_X)$, $k\in \N$, of uniform quasisymmetries.
\end{theorem}

We remark that Cattalani was able to build on the techniques of \cite{Ntalampekos:qs_approximation} and generalize the above result to arbitrary dimension \cite{Cattalani:qs}. Combining Theorem \ref{theorem:main:approximation} and Theorem \ref{theorem:smooth} one obtains the next corollary.

\begin{corollary}
Every $2$-dimensional quasisphere is the Gromov--Hausdorff limit of a sequence of $2$-dimensional smooth uniform quasispheres that satisfy \ref{smooth:llc-mod} and \ref{smooth:loewner} with uniform parameters. 
\end{corollary}

\subsection{Uniformization of domains}\label{section:qs_surfaces}
We discuss extensions of the Bonk--Kleiner theorem to non-compact surfaces. The proofs of the results in this section rely on the Bonk--Kleiner theorem.

A metric space $(X,d)$ is \textit{linearly locally contractible} if there exists a constant $M\geq 1$ such that every ball $B(a,r)$ in $X$ with $0<r\leq \diam(X)/M$ is contractible inside $B(a,Mr)$. Every closed metric surface that is linearly locally contractible is $\llc$, quantitatively. The converse statement is also true, but it is not quantitative; see \cite{BonkKleiner:quasisphere}*{Lemma 2.5}.

Wildrick in \cite{Wildrick:parametrization} provides a uniformization result for simply connected metric surfaces that are Ahlfors $2$-regular and $\llc$, generalizing the Bonk--Kleiner theorem. Namely, he provides conditions that distinguish between surfaces that are quasisymmetric to the Euclidean $2$-sphere, the once-punctured $2$-sphere, the plane, the upper half-plane, and the unit disk. In \cite{Wildrick:structure} it is shown that locally Ahlfors $2$-regular and locally linearly locally contractible surfaces are locally quasisymmetric to the unit disk. 

More generally, Merenkov and Wildrick \cite{MerenkovWildrick:uniformization} studied the quasisymmetric uniformization of metric spaces that are homeomorphic to planar domains. They sought conditions so that such a metric space is quasisymmetric to a \textit{circle domain} in $\widehat \C$, i.e., a domain each of whose boundary components is a point or a circle. This problem is motivated by \textit{Koebe's conjecture}, which asserts that every domain in the Riemann sphere is conformally equivalent to a circle domain. The conjecture was verified for countably connected domains by He and Schramm \cite{HeSchramm:Uniformization}, but the general case remains open. Merenkov and Wildrick obtained the next result.

\begin{theorem}[\cite{MerenkovWildrick:uniformization}]\label{theorem:merenkovwildrick}
Let $(X,d)$ be an Ahlfors $2$-regular metric space homeomorphic to a finitely connected domain in $\widehat \C$. Then $(X,d)$ is quasisymmetrically equivalent to a circle domain if and only if $(X,d)$ is $\llc$ and its completion is compact. 
\end{theorem}

The statement is quantitative in the sense that the distortion function of the quasisymmetric map may be chosen to depend on the constants associated to the assumptions on $X$ and on the ratio of the diameter of $X$ to the minimum distance between components of $\partial X$.  The result does not hold in the infinitely connected case as was shown in \cite{MerenkovWildrick:uniformization}*{Theorem 1.5}. 

Actually, the main result of \cite{MerenkovWildrick:uniformization} provides sufficient conditions for the uniformization of spaces $X$ homeomorphic to countably connected domains, but we do not state it here for the sake of brevity. We remark that the provided conditions are not quasisymmetrically invariant. Nevertheless, more recently, Rehmert \cite{Rehmert:thesis}*{Theorem 1.0.1} provided quasisymmetrically invariant sufficient conditions so that an Ahlfors $2$-regular and $\llc$ metric space that is homeomorphic to a countably connected domain is quasisymmetric to a circle domain. 

We do not discuss further the uniformization of domains, since this problem is related to another broad and well-studied topic, the uniformization of Sierpi\'nski carpets. The study of this topic is facilitated by a technical and powerful tool called \textit{transboundary modulus}, which was introduced by Schramm \cite{Schramm:transboundary}. We direct the reader to \cites{Bonk:uniformization, BonkMerenkov:rigidity, Ntalampekos:CarpetsThesis, Rehmert:thesis, HakobyanLi:qs_embeddings, Ntalampekos:uniformization_packing} for recent works on the uniformization of carpets and to \cites{HerronKoskela:QEDcircledomains, HeSchramm:Uniformization, Schramm:transboundary, HildebrandtMosel:uniformization, Bonk:square, SolyninVidanage:rectangular, Rajala:koebe, NtalampekosRajala:exhaustion, EsmayliRajala:quasitripod,  KarafylliaNtalampekos:gromov_hyperbolic, LiRajala:cofat} for some recent results on the uniformization of domains.

%Soon later, analogous results were obtained for compact surfaces.

%\begin{theorem}[\cites{GeyerWildrick:uniformization, Ikonen:isothermal}]
%Let $(X,d)$ be a compact metric surface that is Ahlfors $2$-regular and linearly locally contractible. Then there exists a complete Riemannian metric $g$ on $X$ of constant curvature such that the identity map from $(X,d)$ onto $(X,g)$, is quasisymmetric with distortion function depending only on the data of $X$.
%\end{theorem}

%Here the data of $X$ refer to the constants in the Ahlfors $2$-regularity and the linear local contractibility conditions. The initial result of Geyer and Wildrick restricted to orientable surfaces and involved dependence on the topological genus of $X$. These requirements were removed by Ikonen. As we discussed before, the result of Fitzi--Meier in Theorem \ref{theorem:fitzimeier} provides canonical quasisymmetric parametrizations of surfaces as in the above theorem.

\subsection{Uniformization via weak metric doubling measures}\label{section:weakly_doubling_measures}

David and Semmes in \cites{DavidSemmes:weights, DavidSemmes:fractals} studied metric doubling measures, which are used to deform a given metric in a controlled way. Let $(X,d)$ be a metric space. A Borel measure $\mu$ on $X$ is a \textit{metric doubling measure of dimension $2$} if there exists a metric $q$ on $X$ and a constant $C\geq 1$ such that for all $x,y\in X$ we have
$$C^{-1}\mu(B_d(x,d(x,y)))^{1/2} \leq q(x,y)\leq C \mu(B_d(x,d(x,y)))^{1/2}.$$
The existence of such a measure implies that the identity map from $(X,d)$ onto $(X,q)$ is quasisymmetric. Lohvansuu, Rajala, and Rasimus \cite{LohvansuRajalaRasimus:qs_wmdm} introduce the notion of a weak metric doubling measure, which, roughly speaking, is required to satisfy only the left of the above inequalities. Then they use that notion to characterize quasispheres.

\begin{theorem}[\cite{LohvansuRajalaRasimus:qs_wmdm}]
%Let $(X,d)$ be a metric space homeomorphic to a finitely connected domain in $\widehat \C$. Then $(X,d)$ is quasisymmetrically equivalent to a circle domain if and only if $(X,d)$ is $\llc$, carries a weakly doubling measure of dimension $2$, and its completion is compact. 
A metric $2$-sphere $(X,d)$ is a quasisphere if and only if $X$ is $\llc$ and carries a weak metric doubling measure of dimension $2$.
\end{theorem}

We now give the formal definition of a weak metric doubling measure. Let $(X,d)$ be a metric space. A Borel measure $\mu$ on $X$ is \textit{doubling} if there exists a constant $C\geq 1$ such that
$$0<\mu(B(x,2r)) \leq C \mu (B(x,r))<\infty$$
for each $x\in X$ and $r>0$. For $x,y\in X$, let $B_{xy}=B(x,d(x,y))\cup B(y,d(x,y))$. Let $\mu$ be a doubling measure on $X$. For $s>0$, $x,y\in X$, and $\delta>0$, let
$$q_{\mu,s}^\delta(x,y)=\inf\left\{\sum_{j=1}^m \mu(B_{x_{j-1}x_j})^{1/s}\right\}$$
where the infimum is taken over all chains of points $x=x_0,x_1,\dots,x_m=y$, $m\in \N$, such that $d(x_{j-1},x_j)\leq \delta$ for $j\in \{1,\dots,m\}$. Then one defines
$$q_{\mu,s}(x,y)=\limsup_{\delta\to 0^+} q_{\mu,s}^\delta(x,y).$$
A doubling measure $\mu$ on $X$ is a \textit{weak metric doubling measure of dimension $s$} if there exists a constant $C\geq 1$ such that for all $x,y\in X$ we have
$$ C^{-1} \mu(B_{xy})^{1/s}\leq q_{\mu,s}(x,y).$$
The main element in the proof of the above theorem is to show that the $\llc$ condition implies that the reverse inequality is true when $s=2$. Eventually, the authors do resort to the Bonk--Kleiner theorem to obtain a quasisymmetry from $(X,q_{\mu,2})$ onto $\widehat \C$.

Weak doubling measures are also used by Rajala and Rasimus \cite{RajalaRasimus:qs_koebe} in the problem of quasisymmetric uniformization of metric spaces homeomorphic to a finitely connected domain. Specifically, they generalize the theorem of Merenkov and Wildrick (Theorem \ref{theorem:merenkovwildrick}).

\section{Quasiconformal uniformization}\label{section:qc}

\subsection{Motivation and examples}
Let $X$ be a metric surface of locally finite Hausdorff $2$-measure. Recall that for a set $G \subset X$ and disjoint subsets $E,F\subset G$ we denote by $\Gamma(E,F;G)$ the family of curves in $G$ joining $E$ and $F$. %The endpoints of a curve $\gamma\in \Gamma(E,F;G)$ are allowed to lie in $\partial G$, but otherwise, the trace of $\gamma$ is contained in $G$. 
A \textit{(topological) quadrilateral} in $X$ is a closed Jordan region $Q$ together with a partition of $\partial Q$ into four edges $\zeta_1,\zeta_2,\zeta_3,\zeta_4\subset \partial Q$ enumerated in cyclic order that are non-overlapping, i.e., they can only intersect at the endpoints. When we refer to a quadrilateral $Q$, it will be implicitly understood that there exists such a partition of its boundary. We define $\Gamma(Q)=\Gamma(\zeta_1,\zeta_3;Q)$  and $\Gamma^*(Q) =\Gamma(\zeta_2,\zeta_4;Q)$. 

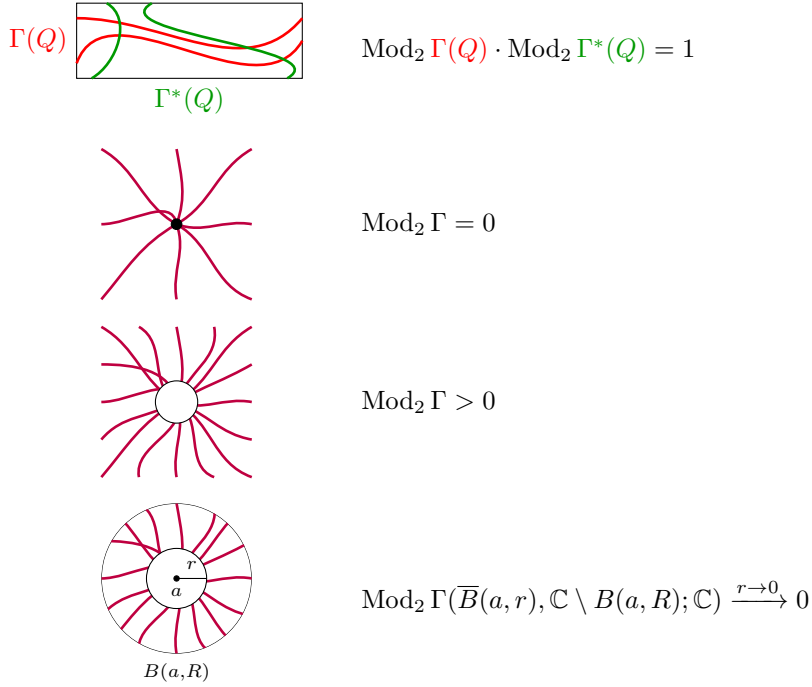
\begin{figure}
\begin{minipage}{0.37\textwidth}
\begin{tikzpicture}
\draw (0,0) rectangle (3,1);
\node[anchor=east,red] at (0,0.5) {$\Gamma(Q)$};
\draw[red, line width=1] (0,0.2) to[out=70, in=-120]  (3,0.5);
%\draw[red, line width=1] (0,.06) to[out=0, in=120]  (3,0.1);
\draw[red, line width=1] (0,0.8) to[out=0, in=-130]  (3,0.8);

\node[anchor=north, color=green!60!black] at (1.5,0) {$\Gamma^*(Q)$};
\draw[color=green!60!black, line width=1] (0.2,0) to[out=30, in=-30] (0.4, 1);
%\draw[color=green!60!black, line width=1] (0.8,0) to[out=150, in=-30] (2, 1);
\draw[color=green!60!black, line width=1] (2.8,0) to[out=30, in=-150] (1, 1);
\end{tikzpicture}
\end{minipage}
\begin{minipage}{0.62\textwidth}
$\Mod_2 \textcolor{red}{\Gamma(Q)}\cdot \Mod_2 \textcolor{green!60!black}{\Gamma^*(Q)} =1$

\end{minipage}

\bigskip

\begin{minipage}{0.37\textwidth}
\centering
\begin{tikzpicture}
\draw[color=purple, line width=1] (0,0) to[out=100,in=0] (-1,0);
\draw[color=purple, line width=1] (0,0) to[out=150,in=-30] (-1,1);
\draw[color=purple, line width=1] (0,0) to[out=70,in=-80] (0,1);
\draw[color=purple, line width=1] (0,0) to[out=30,in=-150] (1,1);
\draw[color=purple, line width=1] (0,0) to[out=-20,in=170] (1,0);
\draw[color=purple, line width=1] (0,0) to[out=-30,in=160] (1,-1);
\draw[color=purple, line width=1] (0,0) to[out=-70,in=100] (0,-1);
\draw[color=purple, line width=1] (0,0) to[out=-150,in=50] (-1,-1);
\draw[fill=black] (0,0) circle (2pt);
\end{tikzpicture}
\end{minipage}
\begin{minipage}{0.62\textwidth}
$\Mod_2\Gamma=0$
\end{minipage}

\bigskip

\begin{minipage}{0.37\textwidth}
\centering
\begin{tikzpicture}
\draw[color=purple, line width=1] (0,0) to[out=100,in=0] (-1,0);
\draw[color=purple, line width=1] (0,0) to[out=150,in=-30] (-1,1);
\draw[color=purple, line width=1] (0,0) to[out=70,in=-80] (0,1);
\draw[color=purple, line width=1] (0,0) to[out=30,in=-150] (1,1);
\draw[color=purple, line width=1] (0,0) to[out=-20,in=170] (1,0);
\draw[color=purple, line width=1] (0,0) to[out=-30,in=160] (1,-1);
\draw[color=purple, line width=1] (0,0) to[out=-70,in=100] (0,-1);
\draw[color=purple, line width=1] (0,0) to[out=-150,in=50] (-1,-1);

\draw[color=purple, line width=1] (0,0) to[out=100,in=0] (-1,0.5);
\draw[color=purple, line width=1] (0,0) to[out=150,in=-30] (-0.5,1);
\draw[color=purple, line width=1] (0,0) to[out=70,in=-80] (0.5,1);
\draw[color=purple, line width=1] (0,0) to[out=30,in=-150] (1,0.5);
\draw[color=purple, line width=1] (0,0) to[out=-20,in=170] (1,-0.5);
\draw[color=purple, line width=1] (0,0) to[out=-30,in=160] (0.5,-1);
\draw[color=purple, line width=1] (0,0) to[out=-70,in=100] (-0.5,-1);
\draw[color=purple, line width=1] (0,0) to[out=-150,in=50] (-1,-0.5);
\draw[fill=white] (0,0) circle (8pt);
\end{tikzpicture}
\end{minipage}
\begin{minipage}{0.62\textwidth}
$\Mod_2\Gamma>0$
\end{minipage}

\bigskip

\begin{minipage}{0.37\textwidth}
\centering
\begin{tikzpicture}
\node[below] at (0,-1) {$\scriptstyle B(a,R)$};
\clip (0,0) circle (1cm);
\draw (0,0) circle (1cm);

\draw[color=purple, line width=1] (0,0) to[out=100,in=0] (-1,0);
\draw[color=purple, line width=1] (0,0) to[out=150,in=-30] (-1,1);
\draw[color=purple, line width=1] (0,0) to[out=70,in=-80] (0,1);
\draw[color=purple, line width=1] (0,0) to[out=30,in=-150] (1,1);
\draw[color=purple, line width=1] (0,0) to[out=-20,in=170] (1,0);
\draw[color=purple, line width=1] (0,0) to[out=-30,in=160] (1,-1);
\draw[color=purple, line width=1] (0,0) to[out=-70,in=100] (0,-1);
\draw[color=purple, line width=1] (0,0) to[out=-150,in=50] (-1,-1);

\draw[color=purple, line width=1] (0,0) to[out=100,in=0] (-1,0.5);
\draw[color=purple, line width=1] (0,0) to[out=150,in=-30] (-0.5,1);
\draw[color=purple, line width=1] (0,0) to[out=70,in=-80] (0.5,1);
\draw[color=purple, line width=1] (0,0) to[out=30,in=-150] (1,0.5);
\draw[color=purple, line width=1] (0,0) to[out=-20,in=170] (1,-0.5);
\draw[color=purple, line width=1] (0,0) to[out=-30,in=160] (0.5,-1);
\draw[color=purple, line width=1] (0,0) to[out=-70,in=100] (-0.5,-1);
\draw[color=purple, line width=1] (0,0) to[out=-150,in=50] (-1,-0.5);
\draw[fill=white] (0,0) circle (0.4cm);
\draw[fill=black] (0,0)node[below]{$\scriptstyle a$} circle (1pt) -- (0.4,0) node[pos=0.5,anchor=south] {$\scriptstyle r$};

\end{tikzpicture}
\end{minipage}
\begin{minipage}{0.62\textwidth}
$\displaystyle{\Mod_2\Gamma(\br B(a,r), \C \setminus B(a,R);\C) \xrightarrow[]{r\to 0} 0}$
\end{minipage}

\caption{Elementary properties of modulus in the plane.}\label{fig:modulus_plane}
\end{figure}

In the case of a quadrilateral $Q$ in the plane, we always have (see \cite{LehtoVirtanen:quasiconformal}*{\S I.4})
$$\Mod_2\Gamma(Q) \cdot \Mod_2\Gamma^*(Q) =1.$$
Moreover, for every $a\in \C$ and $0<r<R$ we have (see \cite{LehtoVirtanen:quasiconformal}*{I.6})
\begin{align*}
\Mod_2\Gamma(\overline B(a,r), \C\setminus B(a,R); \C)= 2\pi \left(\log\frac{R}{r}\right)^{-1}.
\end{align*}
In particular, this quantity converges to $0$ as $r\to 0$. As a consequence, the family of non-constant curves passing through a given point of $\C$ has modulus zero. See Figure \ref{fig:modulus_plane} for an illustration of the elementary properties of modulus in the plane.

Rajala \cite{Rajala:uniformization} studied the problem of quasiconformal uniformization of metric surfaces. 

\begin{definition}\label{definition:qc}
Let $X,Y$ be metric surfaces of locally finite Hausdorff $2$-measure. A homeomorphism $f\colon X\to Y$ \textit{quasiconformal} if there exists $K\geq 1$ such that for every family of curves $\Gamma$ in $X$ we have
$$K^{-1}\Mod_2\Gamma \leq \Mod_2 f(\Gamma) \leq K\Mod_2\Gamma.$$
Here $f(\Gamma)$ denotes the family of curves $\{f\circ \gamma: \gamma\in \Gamma\}$. In that case, $f$ is called $K$-quasiconformal. 
\end{definition}

\begin{remark}
There are three competing definitions of quasiconformality: the geometric, given above in Definition \ref{definition:qc}, the analytic, given below in Theorem \ref{theorem:definitions_qc}, and the metric. Although the geometric and analytic definitions are always equivalent on metric surfaces of locally finite Hausdorff $2$-measure thanks to Theorem  \ref{theorem:definitions_qc}, this is not the case with the metric definition. Sufficient conditions for the equivalence of the metric definition with the other definitions were provided by Heinonen and Koskela \cite{HeinonenKoskela:qc}. 

However, it was observed by Rajala, Rasimus, and Romney \cite{RajalaRasimusRomney:uniformization}*{Lemma 5.5} that the existence of a metrically quasiconformal parametrization of a surface by Euclidean space imposes some strong geometric restrictions, like \textit{infinitesimal linear local connectedness}. In addition, in \cite{RajalaRasimusRomney:uniformization}*{Section 5}, one can find examples of surfaces with good geometry (e.g., satisfying the upper bound of Ahlfors $2$-regularity or the $\llc$ condition) that admit geometrically quasiconformal parametrizations, but do not admit metrically quasiconformal ones by Euclidean space. 

The line of research discussed in this survey asks for minimal conditions for the parametrization of metric surfaces by Euclidean space (recall Question \ref{question:minimal}). In Section \ref{section:wqc} we present a result that allows the \textit{weakly} quasiconformal parametrization, according to the geometric definition using modulus, of \textit{any} surface of locally finite Hausdorff $2$-measure by Euclidean space, without imposing any further geometric condition. It is clear, therefore, that the ``right" definition of quasiconformality in our setting is the geometric, or equivalently, the analytic one. 
\end{remark}

In order for a surface to be quasiconformally equivalent to the plane it is necessary that modulus on that surface satisfies the conditions described in Figure \ref{fig:modulus_plane} up to constants. Rajala observed the following two phenomena, which obstruct quasiconformal parametrization of a surface by the plane.

\begin{example}\label{example:collapse}
Let $B$ be a closed ball in $\C$ and let $X$ be the quotient metric space obtained from $\C$ by identifying points in $B$, while all other points of $\C$ have trivial equivalence classes. See Figure \ref{figure:collapse} for an illustration of the space $X$, assuming it is embedded in Euclidean $3$-dimensional space. The space $X$ is equipped with the quotient metric. The ball $B$ projects to a point $p\in X$. The space $X$ is homeomorphic to $\C$ and in fact locally isometric to $\C$ away from the point $p$. Moreover, $X$ has locally finite Hausdorff $2$-measure, since the natural projection from $\C$ onto $X$ is $1$-Lipschitz. It can be shown that the $2$-modulus of the family of non-constant curves in $X$ passing through the point $p$ is non-zero; this is because the modulus of non-constant curves passing through the ball $B$ in $\C$ is non-zero. On the other hand, in the complex plane the modulus of non-constant curves passing through a point is always zero. Therefore, there is an obstruction to the existence of a quasiconformal map from a subset of $\C$ onto $X$. 

There exists a much more involved example \cite{NtalampekosRomney:nonlength}*{Example 8.4}  of a metric surface $X$ of locally finite Hausdorff $2$-measure with the following property. There is a non-degenerate continuum $E\subset X$ such that for each $x\in E$ the conformal modulus of non-constant curves passing through $x$ is non-zero. 
\end{example}

\begin{figure}
\includegraphics[scale=.5]{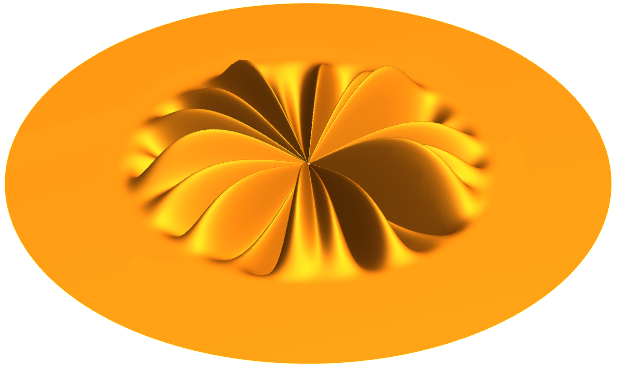}
\caption{A surface obtained by collapsing a ball in the plane to a point. One can think of cutting a round hole in a piece of fabric and then sewing the hole.}\label{figure:collapse}
\end{figure}

\begin{example}\label{example:cantor}
Let $E\subset \C$ be a totally disconnected compact set. For $z,w\in \C$, define
$$d(z,w)= \inf_\gamma \int_\gamma \chi_{\C\setminus E}\, ds$$
where the infimum is taken over all rectifiable curves $\gamma$ in $\C$ joining $z$ and $w$. It is elementary to show that $(\C,d)$ is a length space with locally finite Hausdorff $2$-measure and that the identity map from $(\C,|\cdot |)$ onto $(\C,d)$ is a $1$-Lipschitz homeomorphism that is a local isometry in $\C\setminus E$; see \cite{NtalampekosRomney:length}*{Example 8.3} for an argument. Rajala proved that if $E$ has positive Lebesgue measure, then there exists no quasiconformal homeomorphism between $(\C,d)$ and a subset of the plane \cite{Rajala:uniformization}*{Example 2.1}. 
%Notably, there exists a set $E$ of positive Lebesgue measure such that the space $(\C,d)$ satisfies condition \eqref{ireciprocality:3}; see \cite{NtalampekosRomney:nonlength}*{Example 8.3}. 
See also \cite{IkonenRomney:removable}*{Remark 3.13} for a generalization of this construction. We provide the details of Rajala's argument here. 

We first show that $\mathcal H^2_d(E)=0$. Let $\varepsilon \in (0,1)$. Let $x\in E$ be a Lebesgue density point of $E$. For all sufficiently small $r>0$ we have $|Q(x,r)\setminus  E|< \varepsilon^2 |Q(x,r)|$, where $Q(x,r)$ denotes the square of side length $r$ that is centered at $x$ and has sides parallel to the coordinate axes. For each rectangle $R\subset Q$ with sides parallel to those of $Q$ and with dimensions $\varepsilon r$ and $r$, by an application of Fubini's theorem, we may find a line segment $\gamma$ of length $r$ joining opposite sides of $R$ such that 
$$\int_{\gamma} \chi_{\C\setminus E}\, ds < \varepsilon r.$$
This implies that any two points of $Q(x,r)$ can be joined by a polygonal curve $\gamma$ satisfying
$$\int_{\gamma} \chi_{\C\setminus E}\, ds < 4\varepsilon r.$$
Therefore, for each density point $x\in E$ and for all small $r>0$ we have
\begin{align}\label{example:cantor:ineq}
\diam_d(B(x,r/2))\leq \diam_d(Q(x,r)) \leq 4\varepsilon r.
\end{align}

Suppose that $E$ is contained in an open ball $B_0$ and let $\delta>0$. By the basic covering theorem \cite{Heinonen:metric}*{Theorem 1.2}, we can cover the Lebesgue density points $D$ of $E$ by countably many balls $B(x_i,r_i/2)\subset B_0$, $r_i<\delta$, $i\in \N$, satisfying the inequality \eqref{example:cantor:ineq} and so that the balls $B(x_i,r_i/10)$, $i\in \N$, are disjoint. We have
$$\sum_{i\in \N} \diam_d(B(x_i,r_i/2))^2 \leq \sum_{i\in \N}16\varepsilon^2r^2_i = C \varepsilon^2\sum_{i\in \N} |B(x_i,r_i/10)|\leq C \varepsilon^2\mathcal |B_0|,$$
where $C=1600\pi^{-1}$. As $\delta\to 0$, we obtain $\mathcal H^2_d(D)\leq C\varepsilon^2\mathcal |B_0|$. As $\varepsilon\to 0$, we have $\mathcal H^2_d(D)=0$. Since the identity map from $(\C,|\cdot|)$ onto $(\C,d)$ is $1$-Lipschitz, we also have $\mathcal H^2_d(E\setminus D)=|E\setminus D|=0$. This completes the proof of the claim. Furthermore, given that the identity map from $(\C,|\cdot|)$ onto $(C,d)$ is locally isometric in $\C\setminus E$, we have
\begin{align}\label{example:cantor:measure}
\int g\, d\mathcal H^2_d= \int g \chi_{\C\setminus E}\, d\mathcal H^2 
\end{align}
for each Borel function $g\colon \C\to [0,\infty]$; here the latter integration is with respect to Lebesgue measure.

Let $Q\subset \C$ be an arbitrary closed square with sides parallel to the coordinate axes. Let $\Gamma(Q)$ (resp.\ $\Gamma^*(Q)$) be the family of curves in $Q$ joining the horizontal (resp.\ vertical) sides of $Q$. We will estimate the modulus $\Mod_2\Gamma(Q)$ in the metric $d$ from below. Let $\rho\colon \C\to [0,\infty]$ be a Borel function that is admissible for $\Gamma(Q)$. That is, 
$$\int_{\gamma} \rho\, ds_d\geq 1$$
for all $\gamma\in \Gamma(Q)$; here $ds_d$ denotes the length element in the space $(\C,d)$. By \cite{NtalampekosRomney:length}*{Proposition 8.1 (iii)}, the above integral is equal to $\int_{\gamma} \rho \chi_{\C\setminus E}\, ds$ whenever the Euclidean length of $\gamma$ is finite. Assuming that the side length of $Q$ is $r$ and integrating over all horizontal segments in $\Gamma(Q)$, we obtain
$$ r \leq \int \rho \chi_{{Q\setminus E}}\, d\mathcal H^2\leq \left(\int \rho^2\chi_{Q\setminus E} \, d\mathcal H^2 \right)^{1/2} \mathcal |Q\setminus E|^{1/2}= \left(\int \rho^2 \, d\mathcal H^2_d\right)^{1/2} |Q\setminus E|^{1/2},$$
where the last equality follows form \eqref{example:cantor:measure}.
Therefore, 
\begin{align*}
\Mod_2\Gamma(Q) \geq \frac{|Q|}{|Q\setminus E|}.
\end{align*}
By symmetry, the same estimate is true for $\Mod_2\Gamma^*(Q)$, so
$$\Mod_2\Gamma(Q)\cdot \Mod_2\Gamma^*(Q) \geq \frac{|Q|^2}{|Q\setminus E|^2}.$$
As $Q$ shrinks to a Lebesgue density point of $E$, the above product of moduli converges to $\infty$. On the other hand, if $(\C,d)$ were quasiconformally equivalent to a planar domain, then that product would have to stay uniformly bounded; see the first property that is illustrated in Figure \ref{fig:modulus_plane}. Therefore, there exists no such quasiconformal map.
\end{example}

\subsection{Rajala's theorem}

Rajala \cite{Rajala:uniformization} introduced the notion of a reciprocal surface by excluding the phenomena described in the above examples.

\begin{definition}
A metric surface $X$ of locally finite Hausdorff $2$-measure is \textit{reciprocal} if there exists a constant $\kappa \geq 1$ such that
\begin{align}\label{ireciprocality:12}
    \kappa^{-1}\leq \Mod_2 \Gamma(Q) \cdot \Mod_2\Gamma^*(Q) \leq \kappa \quad \textrm{for each quadrilateral $Q\subset X$}
\end{align}
and 
\begin{align}\label{ireciprocality:3}
        &\lim_{r\to 0} \Mod_2 \Gamma( \br B(a,r), X\setminus B(a,R);X )=0 \quad \textrm{for each ball $B(a,R)$.} 
\end{align}
A metric surface $X$ is \textit{locally reciprocal} if each point has an open neighborhood that is reciprocal. 
\end{definition}

See also Theorem \ref{theorem:reciprocal_simplify} below for a simplification of this definition. The main result of Rajala \cite{Rajala:uniformization} shows that reciprocal surfaces are precisely those admitting a quasiconformal parametrization by Euclidean space.

\begin{theorem}[Rajala's uniformization theorem]\label{theorem:rajala}
%Let $X$ be a metric surface of locally finite Hausdorff $2$-measure that is homeomorphic to $\C$. Then there exists a quasiconformal homeomorphism from an open subsets of $\C$ onto $X$ if and only if $X$ is reciprocal. 

Let $X$ be a simply connected metric surface without boundary that has locally finite Hausdorff $2$-measure and is reciprocal. If $X$ is compact, then there exists a quasiconformal homeomorphism from $\widehat{\C}$ onto $X$ and if $X$ is not compact, then there exists a quasiconformal homeomorphism from either $\D$ or $\C$ onto $X$.
\end{theorem}

The proof of Rajala is entirely self-contained and involves the construction of the desired quasiconformal map from scratch without resorting to any other uniformization result.

\begin{proof}[Outline of proof]
Rajala proves that each quadrilateral $Q=Q(\zeta_1,\zeta_2,\zeta_3,\zeta_4)\subset X$ can be mapped quasiconformally to a rectangle in the plane. Then he uses a standard exhaustion argument to map the space $X$ onto $\widehat{\C}$, $\C$, or $\D$. We describe the steps for the uniformization of a quadrilateral $Q$.

The first step in the proof is to find a Borel function $\rho\colon Q\to [0,\infty]$ that is \textit{weakly} admissible for $\Gamma(Q)$ and such that $\int \rho^2\, d\mathcal H^2=\Mod_2\Gamma(Q)\eqqcolon M$. Using $\rho$, one defines a function $u$ on $Q$ by the formula (disregarding some technicalities)
$$u(x)= \inf_{\gamma_x} \int_{\gamma_x} \rho\, ds$$
where the infimum is taken over all rectifiable curves $\gamma_x$ joining $\zeta_1$ to $x$. The function $u$ is continuous, satisfies $u|_{\zeta_1}=0$, $u|_{\zeta_3}=1$, lies in the Sobolev space $N^{1,2}(Q)$ with $\rho$ as a $2$-weak upper gradient, and satisfies the maximum principle. In fact, if $Q$ were a subset of $\C$ then the function $u$ would be a harmonic function equal to the real part of a conformal map $f$ from $Q$ onto the rectangle $[0,1]\times [0,M]$. We remark that in the subsequent work  \cite{RajalaRomney:reciprocal} it is shown that even without the reciprocity assumption on $X$ the function $u$ exists and has the above properties. 

The next step is to define a ``harmonic conjugate" function $v\colon Q\to [0,M]$, by integrating $\rho$ over the level sets of $u$, which are shown to be Jordan arcs connecting $\zeta_2$ and $\zeta_4$. The study of the level sets of $u$ in the paper of Rajala is technical, but there have been some recent developments that provide very general results for the level sets of Sobolev functions on surfaces \cites{Ntalampekos:monotone, EsmayliIkonenRajala:coarea, MeierNtalampekos:rigidity}.

Then properties of $v$ analogous to those of $u$ are established. However, unlike $u$, the definition of $v$ and its properties rely crucially on the reciprocity assumption. The next step is to define the map $f=(u,v)\colon Q\to [0,1]\times [0,M]$ and show that it is a homeomorphism. The proof of the last claim utilizes a change of variables formula for $f$, where $\rho^2$ plays the role of the Jacobian.

One of the most technical parts of the paper is to establish the Sobolev regularity of $f$, and specifically to show that a certain multiple of $\rho$ is a $2$-weak upper gradient of $f$. This is achieved through the study of a modification of conformal modulus, called variational modulus, which is exactly dual to conformal modulus. This step also depends crucially on the reciprocity assumption.  

The Sobolev regularity of $f$ in turn implies that there exists a constant $K\geq 1$ such that $\Mod_2\Gamma \leq K \Mod_2 f(\Gamma)$ for every curve family $\Gamma$ in $Q$.
Finally, it is shown that this inequality implies an analogous reverse inequality, yielding the quasiconformality of $f$.
\end{proof}

Romney \cite{Romney:reciprocal} (see also \cite{Rajala:uniformization}*{Section 14}) proved that if there exists such a quasiconformal map, then it can be taken to satisfy the inequalities
\begin{align}\label{qc:optimal}
\frac{\pi}{4} \Mod_2\Gamma\leq \Mod_2 f(\Gamma) \leq \frac{\pi}{2} \Mod_2\Gamma.
\end{align}
These are the best possible constants and are attained by the identity map from $(\C,|\cdot|)$ onto $(\C, |\cdot|_{\infty})$; see \cite{Rajala:uniformization}*{Example 2.2}. 

The theorem of Rajala was generalized to arbitrary surfaces by Ikonen \cite{Ikonen:isothermal}. The formulation of Ikonen's result below is taken from \cite{NtalampekosRomney:nonlength}*{Theorem 1.6}.

\begin{theorem}\label{theorem:reciprocal_local}
Let $X$ be a metric surface (with boundary) of locally finite Hausdorff $2$-measure. Then there exists a complete Riemannian surface $(Z,g)$ of constant curvature that is homeomorphic to $X$ and a quasiconformal map $f\colon Z \to X$ if and only if
X is locally reciprocal. In that case, $f$ can be taken to satisfy \eqref{qc:optimal}.
\end{theorem}

\subsection{Reciprocal surfaces}\label{section:reciprocal}
We discuss some further properties of reciprocal surfaces.

\begin{theorem}[\cite{Rajala:uniformization}*{Theorem 1.6}, \cite{RajalaRasimusRomney:uniformization}*{Proposition 3.9}]\label{theorem:area_reciprocal}
Let $X$ be a metric surface of locally finite Hausdorff $2$-measure. Suppose that there exists a constant $C>0$ such that for every $x\in X$ we have
$$\limsup_{r\to 0^+} \frac{\mathcal H^2(B(x,r))}{\pi r^2}\leq C.$$
Then $X$ is reciprocal. 
\end{theorem}

%Observe that an Ahlfors $2$-regular metric $2$-sphere is necessarily reciprocal by the above theorem, and therefore admits a quasiconformal parametrization by the Riemann sphere, as provided by Rajala's theorem (Theorem \ref{theorem:rajala}). Under the $\llc$ condition, this quasiconformal map is also quasisymmetric (see Theorem \ref{theorem:qc_qs} below), so one obtains an alternative proof of the Bonk--Kleiner theorem (Theorem \ref{theorem:bonk_kleiner}) via Rajala's theorem.

We note that the condition of Theorem \ref{theorem:area_reciprocal} holds with $C=1$ at $\mathcal H^2$-a.e.\ point of $X$ \cite{Federer:gmt}*{2.10.19 (5)}. However, as Example \ref{example:collapse} shows, the failure of this condition even at one point can ruin reciprocity.

\begin{question}
Does any of the following conditions imply reciprocity of a surface $X$?
\begin{enumerate}[label=\normalfont(\arabic*)]
\item $\displaystyle{\limsup_{r\to 0^+} \frac{\mathcal H^2(B(x,r))}{\pi r^2}<\infty}$ for each $x\in X$.
\item $\displaystyle{\liminf_{r\to 0^+} \frac{\mathcal H^2(B(x,r))}{\pi r^2}\leq C}$ for each $x\in X$.
\item $\displaystyle{\liminf_{r\to 0^+} \frac{\mathcal H^2(B(x,r))}{\pi r^2}<\infty}$ for each $x\in X$.
\end{enumerate}
\end{question}

Note that the weakest of the above conditions, condition (3), already implies condition \eqref{ireciprocality:3}; see Theorem \ref{theorem:upgrade_homeo} below. So the question is how to establish \eqref{ireciprocality:12}.

Rajala and Romney \cite{RajalaRomney:reciprocal} improved the definition of a reciprocal surface by showing that the lower bound in \eqref{ireciprocality:12} holds for all surfaces of locally finite Hausdorff $2$-measure with a universal constant $\kappa^{-1}$. Later, Eriksson-Bique and Poggi-Corradini \cite{ErikssonBiqueCorradini:duality} found the optimal value of that constant to be $\kappa^{-1}=(\pi/4)^2$. Finally, the current author and Romney \cite{NtalampekosRomney:nonlength} showed that condition \eqref{ireciprocality:3} follows from the upper bound in \eqref{ireciprocality:12}. Summarizing, we have the following result that gives the simplest possible definition of a reciprocal surface; see Theorems 1.8 and 1.10 in \cite{NtalampekosRomney:nonlength}.

\begin{theorem}\label{theorem:reciprocal_simplify}
A metric surface $X$ of locally finite Hausdorff $2$-measure is reciprocal if and only if there exists a constant $\kappa\geq 1$ such that for each quadrilateral $Q\subset X$ we have
$$\Mod_2 \Gamma(Q) \cdot \Mod_2\Gamma^*(Q) \leq \kappa.$$ 
Furthermore, if $X$ has non-empty boundary, then $X$ is reciprocal if and only if the above holds for each quadrilateral $Q\subset \inter (X)$ and \eqref{ireciprocality:3} holds for each ball centered at a boundary point. 
\end{theorem}

We remark that condition \eqref{ireciprocality:3} is weaker than the upper bound in \eqref{ireciprocality:12} and is not enough to guarantee reciprocity; see \cite{NtalampekosRomney:nonlength}*{Proposition 1.11}. It is natural to ask whether reciprocity can be defined in terms of modulus of annuli rather than quadrilaterals.

\begin{problem}
Show that a metric surface $X$ of locally finite Hausdorff $2$-measure is reciprocal if and only if there exists a constant $\kappa\geq 1$ such that for each closed annulus $A\subset X$ bounded by two Jordan curves we have
$$\Mod_2 \Gamma(A) \cdot \Mod_2\Gamma^*(A) \leq \kappa,$$
where $\Gamma(A)$ (resp.\ $\Gamma^*(A)$) denotes the family of curves in $A$ joining (resp.\ separating) the boundary components of $A$.
\end{problem}

Recall Theorem \ref{theorem:smooth}, which gives geometric quantitative characterizations of smooth quasispheres. In fact, the characterizations remain valid for reciprocal spheres.

%\begin{theorem}[\cite{Ntalampekos:qs_approximation}*{Theorem 1.4}]
%Let $X$ be a metric $2$-sphere of finite Hausdorff $2$-measure that is reciprocal. Then the conclusion of Theorem \ref{theorem:smooth} is true. 
%\end{theorem}

\begin{theorem}[\cite{Ntalampekos:qs_approximation}*{Theorems 1.4 and 1.5}]\label{theorem:reciprocal_qs}
Let $X$ be a metric $2$-sphere of finite Hausdorff $2$-measure. 
\begin{enumerate}[label=\normalfont(\arabic*)]
\item If $X$ satisfies \ref{smooth:llc-mod}, then $X$ satisfies \ref{smooth:quasisphere} and is reciprocal.
\item If $X$ satisfies \ref{smooth:quasisphere}, then $X$ satisfies \ref{smooth:loewner}.
\item If $X$ satisfies \ref{smooth:loewner} and is reciprocal, then $X$ satisfies \ref{smooth:llc-mod}.
\end{enumerate}
In particular, if $X$ is reciprocal, then the conclusion of Theorem \ref{theorem:smooth} is true. 
\end{theorem}

Examples of quasispheres of finite Hausdorff $2$-measure that are not reciprocal have been presented by Romney and the author \cites{Romney:absolute, NtalampekosRomney:absolute}; cf.\ Example \ref{example:lusin_qs}. In particular, there exists a quasisphere of finite Hausdorff $2$-measure that does not satisfy condition \ref{smooth:llc-mod}\ref{smooth:modulus}.

\subsection{Quasiconformality and quasisymmetry}
Rajala's uniformization theorem (Theorem \ref{theorem:rajala}) and Theorem \ref{theorem:area_reciprocal} give an alternative proof of the Bonk--Kleiner theorem (Theorem \ref{theorem:bonk_kleiner}) as follows. If a metric sphere $X$ is Ahlfors $2$-regular, as in the Bonk--Kleiner theorem, then Theorem \ref{theorem:area_reciprocal} implies that $X$ is reciprocal. By Theorem \ref{theorem:rajala} there exists a quasiconformal map from $\widehat \C$ onto $X$. The fact that this map is quasisymmetric follows from the next classical result.

\begin{theorem}\label{theorem:qc_qs}
Let $f\colon X\to Y$ be a homeomorphism between bounded metric measure spaces and let $K,M>0$ and $L,Q> 1$. Suppose that the following conditions hold.
\begin{enumerate}[label=\normalfont(\arabic*)]
\item \textup{(Normalizations)} There exist points $a_1,a_2,a_3\in X$ such that for $i\neq j$ we have
$$d_X(a_i,a_j) \geq L^{-1} \diam(X) \quad \text{and}\quad d_Y(f(a_i),f(a_j)) \geq L^{-1}\diam (Y).$$\label{qcqs:1}
\item \textup{(Quasiconformality)} For every curve family $\Gamma$ in $X$ we have
$$\Mod_Q \Gamma \leq K\Mod_Qf(\Gamma).$$\label{qcqs:2}
\item \textup{(Assumptions on $X$)} $X$ is an $L$-doubling and $Q$-Loewner space. \label{qcqs:3}
\item \textup{(Assumptions on $Y$)} $Y$ is an $L$-doubling and $L$-$\llc$ space, and for every ball $B(a,r)\subset Y$ we have
$$\Mod_Q \Gamma(\overline B(a,r), Y\setminus B(a,Lr); Y)<M.$$ \label{qcqs:4}
\end{enumerate}
Then $f$ is $\eta$-quasisymmetric for some distortion function $\eta$ depending only on $K,L,M,Q$ and on the Loewner function of $X$. 
\end{theorem}

Weaker versions of that theorem have appeared repeatedly in the literature \cite{HeinonenKoskela:qc}*{Section 4}, \cite{BonkKleiner:quasisphere}*{Proposition 9.1}, \cite{Rajala:uniformization}*{Section 16}, \cite{LytchakWenger:parametrizations}*{Theorem 2.5}. Hence, for educational purposes and for future reference, we include a proof of this fundamental and general result in this note. We start with a preliminary lemma.

\begin{lemma}\label{lemma:modulus_log}
Let $(Y,d_Y,\mu)$ be a metric measure space. Suppose that there exist constants $M>0$ and  $L,Q>1$ such that for every ball $B(a,r)\subset Y$ we have 
$$\Mod_Q \Gamma(\br B(a,r), Y\setminus B(a,L r);Y) <M.$$
Then for each $a\in Y$, $r>0$ and $R\geq Lr$ we have
$$\Mod_Q\Gamma(\br B(a,r) , Y\setminus B(a,R);Y)\leq C(L,M,Q) \left(\log \frac{R}{r}\right)^{1-Q}.$$
\end{lemma}

\begin{proof}
Let $N\in \N$ such that $L^Nr\leq R$ and $L^{N+1}r>R$. By assumption we have
$$\Mod_Q \Gamma_k<M,$$
where $\Gamma_k=\Gamma( \br B(a,L^{-k}R) , Y\setminus B(a, L^{-k+1}R); Y)<M$
for $k\in \{1,\dots,N\}$. For $k\in \{1,\dots,N\}$, let $\rho_k\colon Y\to[0,\infty]$ be a Borel function that is admissible for $\Gamma_k$ such that $\|\rho\|_{L^Q(Y)}^Q<M$. Let $A_k=B(a,L^{-k+1}R)\setminus \br B(a,L^{-k}R)$ and define 
$$\rho= \frac{1}{N}\sum_{k=1}^N \rho_k\chi_{A_k}.$$
Each curve $\gamma\in \Gamma=\Gamma(\br B(a,r) , Y\setminus B(a,R);Y)$ has a subcurve $\gamma_k\in \Gamma_k$ whose trace lies in $A_k$ except for the endpoints. Thus, $\int_{\gamma_k}\rho_k\chi_{A_k}\, ds\geq 1$ for $k\in \{1,\dots,N\}$ and
$$\int_{\gamma}{\rho}\, ds \geq \frac{1}{N} \sum_{k=1}^N \int_{\gamma_k}\rho_k\chi_{A_k}\, ds\geq 1.$$
This shows that $\rho$ is admissible for $\Gamma$ and
\begin{align*}
\Mod_Q\Gamma &\leq \int \rho^Q\, d\mu= \frac{1}{N^Q} \sum_{k=1}^N \int_{A_k} \rho_k^Q\, d\mu \leq \frac{M}{N^{Q-1}}\\
&\leq M(2\log L)^{Q-1} \left( \log \frac{R}{r}\right)^{1-Q}.
\end{align*}
This completes the proof. 
\end{proof}

\begin{proof}[Proof of Theorem \ref{theorem:qc_qs}]
We will show that $f$ is \textit{weakly quasisymmetric}. That is, there exists a constant $H\geq 1$ depending only on $K,L,M,Q$ and the Loewner function of $X$ such that whenever three points $x_1,x_2,x_3\in X$ satisfy $d_X(x_1,x_2)\leq d_X(x_1,x_3)$ we have 
\begin{align}\label{theorem:qc_qs:weak}
d_Y(f(x_1),f(x_2))\leq Hd_Y(f(x_1),f(x_3)).
\end{align}
A theorem of V\"ais\"al\"a \cite{Heinonen:metric}*{Theorem 10.19} implies that a weakly quasisymmetric map between connected doubling spaces, as assumed in \ref{qcqs:3} and \ref{qcqs:4}, is quasisymmetric, quantitatively, as desired.

We now focus on proving the claim. For $x\in X$ we use the notation $x'=f(x)$. Suppose that 
\begin{align}\label{theorem:qc_qs:assumption}
d_X(x_1,x_2)\leq d_X(x_1,x_3). 
\end{align}
Without loss of generality $x_1\neq x_3$, otherwise \eqref{theorem:qc_qs:weak} holds trivially. By the triangle inequality and assumption \ref{qcqs:1}, there exists at most one index $j\in \{1,2,3\}$ such that $d_Y(x_1',a_j')< 2^{-1}L^{-1}\diam(Y)$. Hence, at least two indices $j\in \{1,2,3\}$ satisfy $d_Y(x_1',a_j')\geq 2^{-1}L^{-1}\diam(Y)$. Among these two indices, we choose one of them that satisfies $d_X(x_2,a_j)\geq 2^{-1}L^{-1}\diam(X)$; again by \ref{qcqs:1} the reverse inequality can only be satisfied by at most one index. Without loss of generality $j=1$, so we have
\begin{align}\label{theorem:qc_qs:normalizations}
d_Y(x_1',a_1') \geq 2^{-1}L^{-1}\diam(Y)\quad \text{and}\quad d_X(x_2,a_1) \geq 2^{-1}L^{-1}\diam(X).
\end{align}

By assumption \ref{qcqs:4}, $Y$ is $L$-$\llc$. Since $x_1',x_3'\in B(x_1',2d_Y(x_1',x_3'))$, there exists a continuum $E'\subset B(x_1',2Ld_Y(x_1',x_3'))$ joining $x_1'$ and $x_3'$. We define $$A=\min\{d_Y(x_1',x_2'), 2^{-1}L^{-1}\diam(Y)\}$$ and note that $x_2',a_1'\notin B(x_1',A)$ in view of \eqref{theorem:qc_qs:normalizations}. Thus there exists a continuum $F'\subset Y\setminus B(x_1', L^{-1}A)$ joining $x_2'$ and $a_1'$. 

If $3Ld_Y(x_1',x_3')> L^{-1}A$, then $d_Y(x_1',x_2') \leq 3L^2 d_Y(x_1',x_3')$, so the desired inequality \eqref{theorem:qc_qs:weak} holds for $H\geq 3L^2$. We assume that $3Ld_Y(x_1',x_3')\leq L^{-1}A$. This implies that the continua $E'$ and $F'$ are disjoint. Let $E=f^{-1}(E')$ and $F=f^{-1}(F')$. The continuum $E$ connects the points $x_1,x_3$ and the continuum $F$ connects the points $x_2,a_1$, so $\diam (F) \geq d_X(x_2,a_1)\geq 2^{-1}L^{-1}\diam(X)$ by \eqref{theorem:qc_qs:normalizations}. Observe that
\begin{align*}
\Delta(E,F)&= \frac{\dist(E,F)}{\min\{\diam(E),\diam(F)\}}\leq \frac{d_X(x_1,x_2)}{\min\{d_X(x_1,x_3), 2^{-1}L^{-1}\diam(X)\}}\leq 2L,
\end{align*}
where the last inequality follows from \eqref{theorem:qc_qs:assumption}.

By assumption \ref{qcqs:3}, $X$ is a $Q$-Loewner space, so there exists a decreasing function $\varphi\colon (0,\infty)\to (0,\infty)$ such that
\begin{align*}
\Mod_Q \Gamma(E,F;X) \geq \varphi( \Delta(E,F))\geq \varphi(2L).
\end{align*}
By assumption \ref{qcqs:2} we have
\begin{align}\label{theorem:qc_qs:phi}
\varphi(2L)\leq \Mod_Q\Gamma(E,F;X) \leq K \Mod_Q\Gamma(E',F';Y).
\end{align}
Every path connecting $E'$ and $F'$ connects $\br B(x_1', 2Ld_Y(x_1',x_3'))$ to $Y\setminus B(x_1',L^{-1}A)$ as well. We assume that $L^{-1}A\geq L\cdot 2Ld(x_1',x_3')$, since otherwise the desired inequality \eqref{theorem:qc_qs:weak} holds for $H\geq 2L^3$. By assumption \ref{qcqs:4} and Lemma \ref{lemma:modulus_log} we have
\begin{align}\label{theorem:qc_qs:mod}
\Mod_Q\Gamma(E',F';Y) \leq C(L,M,Q) \cdot  \left( \log \frac{L^{-1}A}{2Ld_Y(x_1',x_3')}\right)^{1-Q}.
\end{align}
Therefore, by the definition of $A$, \eqref{theorem:qc_qs:phi}, and \eqref{theorem:qc_qs:mod}, for $C=C(L,M,Q)$ we have
\begin{align*}
\frac{d_Y(x_1',x_2')}{d_Y(x_1',x_3')}\leq \frac{2L A}{d_Y(x_1',x_3')}\leq 4L^3 \exp\left( (K^{-1}C^{-1}\varphi(2L))^{1/(1-Q)}\right).
\end{align*}
We finally take $H=\max\{ 4L^3 \exp\left( (K^{-1}C^{-1}\varphi(2L))^{1/(1-Q)} \right), 3L^2,2L^3\}$ to complete the proof.
\end{proof}

We remark that quasisymmetric maps between metric spaces are quasiconformal under some geometric conditions. For example this holds for quasisymmetric maps in Euclidean space or in smooth manifolds \cite{Heinonen:metric}*{Theorem 11.14}. More generally, a result of Tyson \cite{Tyson:qs_qc} states that a quasisymmetric map between locally compact, connected, and Ahlfors $Q$-regular spaces is quasiconformal in the sense that there exists $K\geq 1$ such that
$$K^{-1}\Mod_Q\Gamma \leq \Mod_Q f(\Gamma)\leq K \Mod_Q\Gamma$$
for every curve family $\Gamma$ in $X$. See also \cites{HeinonenKoskela:qc, Tyson:metric_qc, Williams:qc} for further results in the same spirit. Finally, another result that is more relevant to our setting is the following theorem of Tyson.

\begin{theorem}[\cite{Tyson:lusin}*{Proposition 3.13}]\label{theorem:tyson}
Let $X\subset \R^n$, $n\geq 2$, be an open set and $Y$ be a metric space of locally finite Hausdorff $n$-measure. If $f\colon X\to Y$ is a homeomorphism that is locally $\eta$-quasisymmetric for some distortion function $\eta$, then for every curve family $\Gamma$ in $X$ we have
$$\Mod_n\Gamma\leq K(n,\eta) \Mod_n f(\Gamma).$$
\end{theorem}

\subsection{Complex structures on reciprocal surfaces}\label{section:complex_structures}
In Section \ref{section:classical} we defined that a Riemannian metric $g$ on a Riemann surface $X$ is compatible with the complex structure if in local coordinates $g$ has the representation $\alpha|dz|^2$. Here we broaden this definition and define compatibility of an arbitrary metric of locally finite Hausdorff $2$-measure, not necessarily Riemannian, with the complex structure.

\begin{definition}
Let $X$ be a Riemann surface and $d$ be a metric on $X$ that induces its topology and has locally finite Hausdorff $2$-measure. We say that $d$ is \textit{quasicompatible} with the complex structure of $X$ if there exists $K\geq 1$ such that each conformal chart $\varphi$ from an open subset $(U,d)$ of $X$ into $\C$ is $K$-quasiconformal. If $K=1$, we say that $d$ is \textit{compatible} with the complex structure of $X$.  
\end{definition}

This definition agrees with the one given in Section \ref{section:classical} in case the metric $d$ arises from a Riemannian metric. Moreover, in polyhedral and Aleksandrov surfaces, as discussed in Section \ref{section:riemann}, the metric is compatible with the complex structure. %This is because in local coordinates the metric is of the form $e^{u(z)}|dz|$, where $u$ is the difference of two subharmonic functions and $e^u$ is locally integrable on analytic curves.

Observe that if a surface $(X,d)$ admits a complex structure that is quasicompatible with $d$, then $X$ is locally quasiconformally equivalent to the plane, so it is locally reciprocal. %In view of Theorem \ref{theorem:reciprocal_local}, $X$ is actually reciprocal and admits a quasiconformal parametrization by a Riemannian surface. 
Ikonen proved the converse. %Namely, each locally reciprocal surface admits a quasicompatible complex structure.

\begin{theorem}[\cite{Ikonen:isothermal}*{Theorem 1.3}]
Let $(X,d)$ be a locally reciprocal metric surface. Then there exists a complex structure on $X$ that is quasicompatible with $d$. 
\end{theorem}

\begin{remark}\label{remark:reciprocal}
With this point of view, one can think of reciprocal surfaces as the largest class of surfaces that admit local parametrizations by the complex plane so that conformal modulus in the plane is comparable to conformal modulus on the surface. Thus, we obtain an answer to Question \ref{question:structure}. 
\end{remark}

%The next lemma gives the relation between conformality and $1$-quasiconformality of maps between Riemann surfaces.

%\begin{lemma}\label{lemma:conformal_compatible}
%Let $X,Y$ be Riemann surfaces with metrics $d_X,d_Y$, respectively, that are compatible with the complex structures. Then each homeomorphism $h\colon X\to Y$ is conformal if and only if $h\colon (X,d_X)\to (Y,d_Y)$ is $1$-quasiconformal.  
%\end{lemma}

%The proof relies on the fact that conformal maps coincide with $1$-quasiconformal maps in planar domains (see \cite{LehtoVirtanen:quasiconformal}) and on the fact that global quasiconformality of a map between metric surfaces follows from the local quasiconformality, by Theorem \ref{theorem:definitions_qc} below [ref]. 

%Each orientable polyhedral surface $X$ admits  a  complex structure  and becomes a Riemann surface. The complex structure is natural in the sense that it is compatible with the polyhedral metric; see \cite{NtalampekosRomney:length}*{Section 2.5} for more details. 

\section{Weakly quasiconformal uniformization}\label{section:wqc}

\subsection{Uniformization under minimal assumptions}
Recall that not every surface of locally finite area admits a quasiconformal parametrization by a smooth surface as illustrated by Examples \ref{example:collapse} and \ref{example:cantor}. The main result of \cite{NtalampekosRomney:nonlength} shows that \textit{every} metric surface of locally finite area admits a \textit{weakly quasiconformal} parametrization by a smooth surface, therefore completing a series of works on the uniformization of metric surfaces under minimal geometric assumptions. The result was first obtained for \textit{geodesic} surfaces by Meier and Wenger \cite{MeierWenger:uniformization} and for surfaces with a length metric by the author and Romney \cite{NtalampekosRomney:length}.

\begin{theorem}[\cite{NtalampekosRomney:nonlength}*{Theorem 1.2}]\label{theorem:wqc}
Let $X$ be a simply connected metric surface without boundary that has locally finite Hausdorff $2$-measure. If $X$ is compact, then there exists a map $h$ from $\widehat{\C}$ onto $X$ and if $X$ is not compact, then there exists a map $h$ from either $\D$ or $\C$ onto $X$ that satisfies the following conditions.
\begin{enumerate}[label=\normalfont(\arabic*)]
	\item The map $h$ is the uniform limit of homeomorphisms.
	\item For every family of curves $\Gamma$ in $X$ we have
$$\Mod_2\Gamma \leq \frac{4}{\pi}\Mod_2 h(\Gamma).$$
\end{enumerate}
\end{theorem}

As discussed after Theorem \ref{theorem:rajala}, the constant $4/\pi$ is optimal and is attained by the $\ell^\infty$ metric on the plane. The maps $h$ as in the statement are called weakly quasiconformal and we include their general definition below after some topological preliminaries. We remark that the map $h$ is not necessarily a homeomorphism as the next example shows. However, under further assumptions on the space $X$ the map $h$ can be upgraded to a homeomorphism (Theorem \ref{theorem:upgrade_homeo}), to a quasiconformal map (Lemma \ref{lemma:upgrade}), or to a quasisymmetric map (Theorem \ref{theorem:qc_qs}).

\begin{example}
Consider the space $X$ discussed in Example \ref{example:collapse} that arises by collapsing a closed ball $B$ in $\C$ to a point $p\in X$. The natural projection from $\C$ to $X$ satisfies the conclusions of Theorem \ref{theorem:wqc}. However, there is no homeomorphism $h\colon \C\to X$ that satisfies the same conclusions. Indeed, if there existed such a homeomorphism, it would have to be quasiconformal in $\C\setminus h^{-1}(p)$, i.e., it would satisfy both modulus inequalities, as in Definition \ref{definition:qc}; this is true because $X$ is locally isometric to $\C$ away from the point $p$ (see also Lemma \ref{lemma:upgrade} below for a more general statement). The analytic definition of quasiconformality in Theorem \ref{theorem:definitions_qc} implies that a point is a removable singularity and $h$ must be quasiconformal on all of $\C$ (see also \cite{NtalampekosRomney:nonlength}*{Lemma 7.1} for a more general removability statement). However, the modulus of non-constant curves in $X$ passing through the point $p$ is non-zero, a contradiction.  
\end{example}

We say that a point $p$ on a metric surface $X$ of locally finite Hausdorff $2$-measure has \textit{positive capacity} if the modulus of the family of non-constant curves passing through $p$ is non-zero. We pose a problem related to points of positive capacity.

\begin{problem}
Let $X$ be a metric surface without boundary that has locally finite Hausdorff $2$-measure. Let $p\in X$ be a point of positive capacity. Show that there exists a weakly quasiconformal map from $\D$ onto a neighborhood of $p$ that maps a non-degenerate continuum onto $p$.
\end{problem}

We present an outline of the proof of Theorem \ref{theorem:wqc} in the compact case, as given in \cites{NtalampekosRomney:length, NtalampekosRomney:nonlength}.
\begin{proof}[Outline of proof]
Let $X$ be a metric $2$-sphere of finite Hausdorff $2$-measure. The main ingredient in the proof is a result discussed below that allows the approximation of $X$ by polyhedral spheres with controlled geometry; see Theorem \ref{thm:extended_polyhedral_approximation}. Specifically, there exists a sequence $X_n$, $n\in \N$, of polyhedral spheres (equipped with a metric that is locally isometric to the intrinsic metric) that converge to $X$ in the Gromov--Hausdorff sense and such that the area of $X_n$ is uniformly bounded in $n\in \N$. 

The classical uniformization theorem (Theorem \ref{theorem:uniformization_classical}) provides us with a sequence of conformal parametrizations $h_n\colon \widehat{\C}\to X_n$, $n\in \N$, normalized appropriately. The fact that the area of $X_n$ is uniformly bounded implies that the family of maps $h_n$, $n\in \N$, is uniformly equicontinuous. Using the Arzel\`a--Ascoli theorem, it is shown that after passing to a subsequence, $h_n$ converges to a continuous and surjective map $h\colon \widehat \C\to X$ that is the uniform limit of homeomorphisms.  The topological properties of $h$ are straightforward, given that $h$ is a limit of the homeomorphisms $h_n\colon \widehat\C \to X_n$. 

The uniform area bound of $X_n$ implies that the minimal $2$-weak upper gradients $g_{h_n}$ are uniformly bounded in $L^2(\widehat \C)$. A compactness argument in the same spirit as \cite{HeinonenKoskelaShanmugalingamTyson:Sobolev}*{Theorem 7.3.9} shows that the limiting map $h$ lies in the Sobolev space $N^{1,2}(\widehat \C, X)$.

On the other hand, the proof of the modulus inequality in Theorem \ref{theorem:wqc} is quite technical and follows from the most delicate part of the polyhedral approximation theorem (Theorem \ref{thm:extended_polyhedral_approximation}), which is the inequality \eqref{ineq:polyhedral_approx}.
\end{proof}

The proof of Theorem \ref{theorem:wqc} in the case of surfaces with a length metric, as in the setting of \cite{MeierWenger:uniformization} and \cite{NtalampekosRomney:length}, is less technical than in the case of arbitrary metric surfaces. Specifically, the proof of the approximation result of Theorem \ref{thm:extended_polyhedral_approximation} is more elementary for length spaces. It would be interesting to know whether one can reduce the general case to the case of length spaces through a solution of the next problem.

\begin{problem}\label{problem:length}
Let $X$ be a metric surface of locally finite Hausdorff $2$-measure. Show that there exists a length space $Y$ of locally finite Hausdorff $2$-measure and a quasiconformal homeomorphism from $X$ onto $Y$. 
\end{problem}

\subsection{Weakly quasiconformal maps between arbitrary surfaces}
Let $\nu\colon X\to Y$ be a continuous map between topological spaces. The map $\nu$ is \textit{proper} if the preimage of each compact set is compact. The map $\nu$ is \textit{monotone} if the preimage of each point is a continuum. The map $\nu$ is \textit{cell-like} if the preimage of each point is a \textit{cell-like set}, i.e., a set that is contractible in all of its open neighborhoods. In $2$-manifolds without boundary cell-like continua coincide with continua that have a simply connected neighborhood that they do not separate. In the $2$-sphere this condition is equivalent to the condition that the continuum is non-separating. The following result follows from the work of Youngs \cite{Youngs:monotone}; see \cite{NtalampekosRomney:nonlength}*{Theorem 6.3} for the current statement.

\begin{theorem}
Let $\nu\colon X\to Y$ be a continuous and surjective map between compact metric surfaces that are homeomorphic. The following are equivalent.
\begin{enumerate}[label=\normalfont(\arabic*)]
	\item $\nu$ is monotone.
	\item $\nu$ is cell-like.
	\item $\nu$ is the uniform limit of homeomorphisms. 
\end{enumerate}
\end{theorem}

In the non-compact case, we have the Armentrout--Quinn--Siebenmann approximation theorem \cite{Daverman:decompositions}*{Corollary IV.25.1A}.
\begin{theorem}
A continuous, proper, and cell-like map between $2$-manifolds without boundary is the uniform limit of homeomorphisms.
\end{theorem}

\begin{definition}
Let $X,Y$ be metric surfaces of locally finite Hausdorff $2$-measure. A continuous, surjective, proper, and cell-like map $h\colon X\to Y$ is \textit{weakly quasiconformal} if there exists $K>0$ such that for every family of curves $\Gamma$ in $X$ we have
$$\Mod_2\Gamma \leq K\Mod_2 h(\Gamma).$$
In that case, we say that $h$ is weakly $K$-quasiconformal. 
\end{definition}

By the above discussion, a weakly quasiconformal map is the uniform limit of homeomorphisms whenever $X,Y$ are compact and homeomorphic to each other or whenever $X,Y$ have empty boundary.

Weakly quasiconformal maps can be upgraded to (quasiconformal) homeomorphisms under some conditions.
\begin{theorem}[\cite{NtalampekosRomney:length}*{Theorem 7.4}]\label{theorem:upgrade_homeo}
Let $X,Y$ be metric surfaces without boundary and with locally finite Hausdorff $2$-measure and let $h\colon X\to Y$ be a weakly quasiconformal map. If for each $y\in Y$ the modulus of the family of non-constant curves passing through $y$ is zero, then $h$ is a homeomorphism. Moreover, a sufficient condition for this property is that
$$\liminf_{r\to 0^+} \frac{\mathcal H^2(B(y,r))}{r^2}<\infty.$$ 
\end{theorem}

\begin{lemma}[\cite{MeierNtalampekos:rigidity}*{Lemma 2.13}]\label{lemma:upgrade}
Let $X,Y$ be metric surfaces without boundary and with locally finite Hausdorff $2$-measure such that $Y$ is reciprocal. Then every weakly quasiconformal map $h\colon X\to Y$ is quasiconformal, quantitatively.
\end{lemma}

\begin{remark}
Theorem \ref{theorem:wqc} implies that every surface of locally finite Hausdorff $2$-measure admits local parametrizations by the complex plane so that conformal modulus in the plane is smaller than conformal modulus on the surface (times the constant $4/\pi$). On the other hand, in view of Lemma \ref{lemma:upgrade}, for non-reciprocal surfaces we cannot have local parametrizations that shrink modulus. Compare to Remark \ref{remark:reciprocal}.
\end{remark}

In the case of arbitrary surfaces the following uniformization theorem is proved in \cite{NtalampekosRomney:nonlength}, generalizing Theorem \ref{theorem:wqc}. See also \cite{Meier:higher} for an alternative proof with the additional assumption that $X$ is locally geodesic.

\begin{theorem}[\cite{NtalampekosRomney:nonlength}*{Theorem 1.3}]
Let $X$ be a metric surface (with boundary) of locally finite Hausdorff $2$-measure. Then there exists a complete Riemannian surface $(Z,g)$ of constant curvature that is homeomorphic to $X$ and a weakly $(4/\pi)$-quasiconformal map $h\colon Z \to X$. 
\end{theorem}

We note that weakly quasiconformal parametrizations of a surface are not canonical or unique, as illustrated by the next example. 

\begin{example}
There exists a metric surface $X$ of locally finite Hausdorff $2$-measure that is homeomorphic to $\C$ such that there exist two weakly $1$-quasi\-conformal maps, one from $\D$ onto $X$ and one from $\C$ onto $X$ \cite{NtalampekosRomney:length}*{Example 8.4}. Note that if $X$ were reciprocal, by Lemma \ref{lemma:upgrade} these two maps would be quasiconformal homeomorphisms, so there would exist a quasiconformal homeomorphism between $\C$ and $\D$, a contradiction to Liouville's theorem.  Therefore, such a surface $X$ cannot be reciprocal. 
\end{example}

On the other hand, if there exists an essentially unique weakly quasiconformal parametrization of a surface, we expect some strong consequences.

\begin{problem}[cf.\ \cite{CreutzRomney:branch}*{Question 1.6}]
Let $X$ be a metric $2$-sphere of finite Hausdorff $2$-measure. Suppose that any two weakly quasiconformal maps from $\widehat \C$ onto $X$ differ by a quasiconformal map of $\widehat \C$. Then show that $X$ is reciprocal. 
\end{problem}

We discuss briefly the boundary behavior of weakly quasiconformal maps. Recall Cara\-th\'eo\-dory's theorem, which implies that a conformal map from $\D$ onto a Jordan region $\Omega\subset \C$ extends to a homeomorphism of the closures. Ikonen generalized that result to quasiconformal maps on metric surfaces \cite{Ikonen:jordan}. We formulate a slightly more general version of Ikonen's theorem. 

\begin{theorem}\label{theorem:ikonen_wqc}
Let $X$ be a metric surface whose completion $\br X$ is homeomorphic to $\br \D$, the set $\partial X= \br X\setminus X$ is homeomorphic to the unit circle $\mathbb S^1$, and $\br X$ has finite Hausdorff $2$-measure. Let $f\colon \D\to X$ be a map that is weakly $K$-quasiconformal for some $K\geq 1$. Then $f$ extends to a weakly $K$-quasiconformal map $F\colon \br \D \to \br X$. Moreover, if $X$ is reciprocal and each ball centered at a point of $\partial X$ satisfies \eqref{ireciprocality:3}, then $F$ is a quasiconformal homeomorphism.
\end{theorem}

\subsection{Polyhedral approximation of metric surfaces}
We present the main tool that is used in the proof of Theorem \ref{theorem:wqc} and which is interesting in its own right.

\begin{theorem}[\cite{NtalampekosRomney:nonlength}*{Theorem 1.1}] \label{thm:extended_polyhedral_approximation}
Let $X$ be a metric surface (with boundary) of locally finite Hausdorff $2$-measure. There exists a sequence of polyhedral surfaces $\{(X_n,d_{X_n})\}_{n=1}^\infty$ each homeomorphic to $X$, where $d_{X_n}$ is a metric that is locally isometric to the intrinsic metric on $X_n$, such that the following properties hold for an absolute constant $K \geq 1$.  
\begin{enumerate}[label=\normalfont(\arabic*)]
    \item \label{item:main_1} There exists an approximately isometric sequence of maps $f_n \colon X_n \to X$, $n \in \mathbb{N}$. Moreover, for each $n\in \N$, the  map $f_n$ is a proper topological embedding.
    \item \label{item:main_2} For each {compact} set $A \subset X$, 
\begin{align}\label{ineq:polyhedral_approx}
	\limsup_{n \to \infty} \mathcal{H}^2(f_n^{-1}(A)) \leq K \mathcal{H}^2(A).
\end{align}    
    \item\label{item:main_3} There exists an approximately isometric sequence of retractions $R_n \colon X\to f_n(X_n)$, $n\in \N$. 
\end{enumerate} 
If $X$ is a length space, then we may take $d_{X_n}$ to be the intrinsic metric on $X_n$.
\end{theorem}

The next example shows that the constant $K$ cannot be taken to be equal to $1$.
\begin{example}
Let $X$ be the unit square in $\R^2$ with the $\ell^\infty$ metric. Then any sequence $X_n$, $n\in \N$, of polyhedral surfaces satisfying the conclusions of Theorem \ref{thm:extended_polyhedral_approximation} necessarily satisfies
\begin{align*}
\liminf_{n\to\infty} \mathcal H^2(X_n)\geq \frac{4}{\pi} \mathcal H^2(X).
\end{align*}
See \cite{NtalampekosRomney:nonlength}*{Example 8.2} for details.
\end{example}

Given that the $\ell^\infty$ metric attains the optimal constants in \eqref{qc:optimal},  it is natural to pose the following question.

\begin{question}
Can we take $K=4/\pi$ in Theorem \ref{thm:extended_polyhedral_approximation}?
\end{question}

We provide an outline of the proof of Theorem \ref{thm:extended_polyhedral_approximation} in the case that $X$ is a length space, as given in \cite{NtalampekosRomney:length}. The general case is proved in \cite{NtalampekosRomney:nonlength} and involves some severe complications, but overall it follows the same scheme as in the case of length spaces.

\begin{proof}[Outline of proof]
Suppose that $X$ is a metric surface that is a length space without boundary. A result of Creutz and Romney \cite{CreutzRomney:triangulation} implies that for each $\varepsilon>0$ there exists a \textit{geometric triangulation} of $X$ with mesh less than $\varepsilon$. That is, $X$ can be written as a locally finite union of non-overlapping closed Jordan regions $T$, called triangular regions, such that  $\diam(T)<\varepsilon$ and $\partial T$ is the union of three non-overlapping geodesics. Note that a vertex of a triangular region can lie in an edge of another triangular region, so the notion of a triangulation here is different from the classical topological notion; see Figure \ref{figure:triangulation}. 

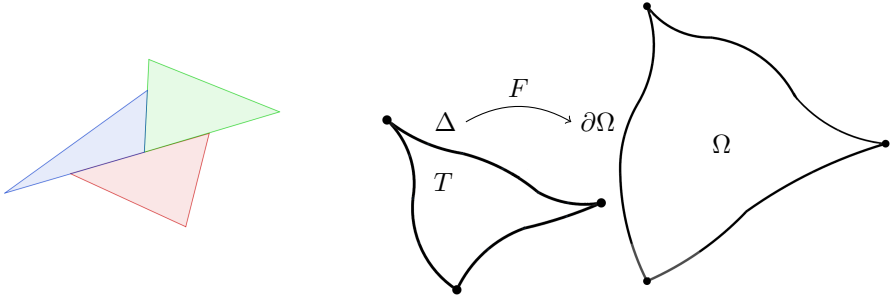
\begin{figure}

\begin{tikzpicture}
\node at (-7,0) {\includegraphics[scale=.2]{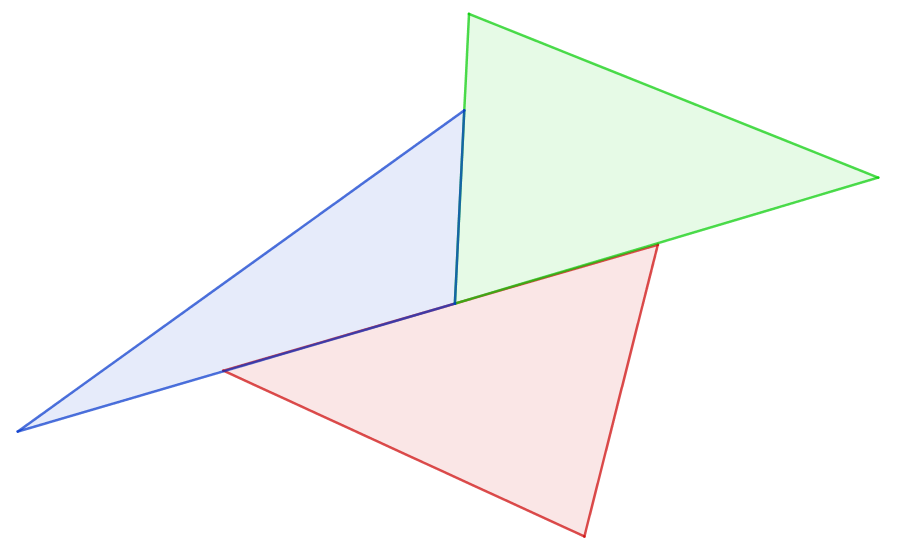}};

    	%\clip (-5,-2.2) rectangle (5,2.5);
           \node () at (-2.3,-0.7) {\includegraphics[width=.25\textwidth]{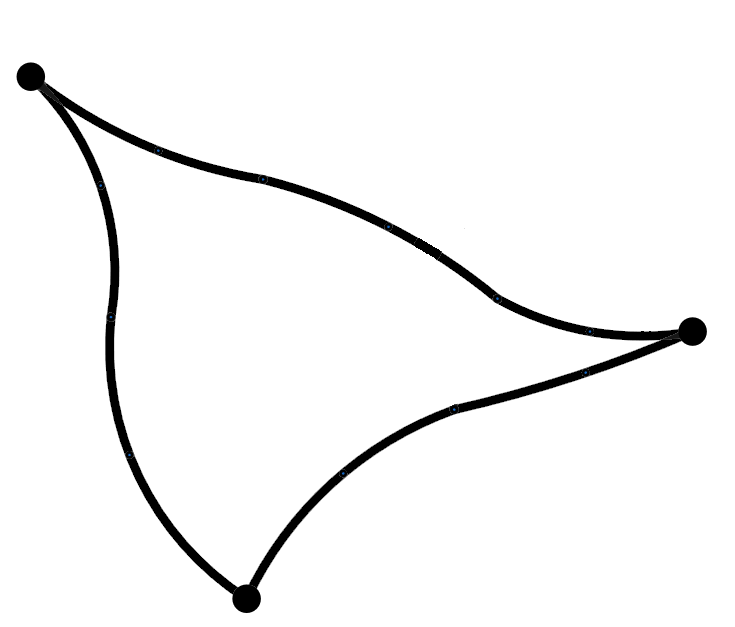}};
           \node at (1.1,0) {\includegraphics[width=.3\textwidth]{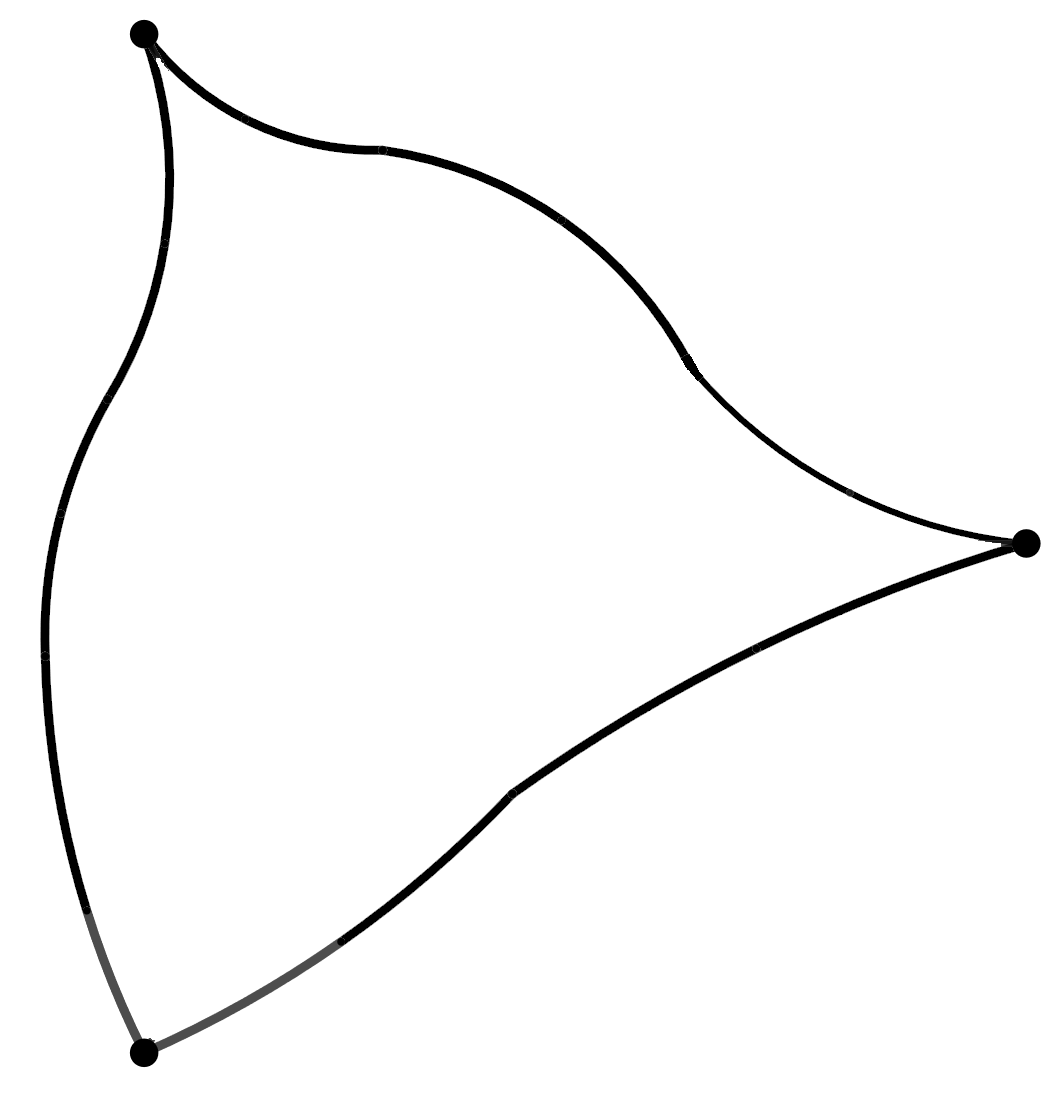}};         
           \draw [->] (-2.7,0.3) node[left] {$\Delta$} to [out=30,in=150]node[pos=0.5,above] {$F$} (-1.3, .3) node[right] {$\partial \Omega$};
           \node at (.7,0) {$\Omega$};
           \node at (-3,-0.5) {$T$};
\end{tikzpicture}
    
\caption{Left: An example of a triangulation given by \cite{CreutzRomney:triangulation}. Right: The embedding of $\Delta=\partial T$ into the plane.}\label{figure:triangulation}
\end{figure}

In order to construct a polyhedral surface $X_n$ that is close to $X$ in the Gromov--Hausdorff sense and satisfies \eqref{ineq:polyhedral_approx} we will replace each triangular region $T$ with a an appropriate polyhedral surface $S$ that has small diameter and most importantly satisfies a bound of the form $\mathcal H^2(S)\leq K\mathcal H^2(T)$.

To achieve that bound, it is first shown that each \textit{metric triangle} $\Delta=\partial T$ (i.e., a metric space homeomorphic to a circle that is equal to the union of three non-overlapping geodesics) admits a $4$-bi-Lipschitz embedding $F$ into the plane \cite{NtalampekosRomney:length}*{Proposition 1.2}. Note that in  general we cannot embed $T$ into the plane via a bi-Lipschitz map, but we are only embedding its boundary! The embedding is given via an explicit formula involving the Gromov product. A sharp bi-Lipschitz distortion bound for this embedding is established in \cite{LuoRomneyTao:triangles}. 

Next, consider the region $\Omega \subset \R^2$ bounded by the bi-Lipschitz embedding of $\Delta=\partial T$; see Figure \ref{figure:triangulation}. A consequence of the Besicovitch inequality for metric spaces (see \cite{Petrunin:lectures_metric_geometry}*{Section 13.D}) implies that $\mathcal H^2(\Omega) \leq C \mathcal H^2(T)$ for a universal $C>0$. See also \cite{NtalampekosRomney:length}*{Theorem 2.1} for an elementary argument. In other words, upon embedding the boundary of $T$ into the plane, one can control the area of the region bounded by the embedded curve $\partial \Omega$ in terms of the area of $T$. 

Finally, we construct a polyhedral surface $S$ by considering a suitable polygonal approximation of $\Omega$ that is scaled appropriately. A feature of the construction is that the area of $S$ is comparable to that of $\Omega$. Then $S$ is attached to $\partial T$ via the embedding $F$ and replaces $T$. By construction we have the desired bound $\mathcal H^2(S)\leq K\mathcal H^2(T)$.
\end{proof}

The approximation result of Theorem \ref{thm:extended_polyhedral_approximation} is generalized to metric $n$-manifolds $X$ with finite Hausdorff $n$-measure by Marti and Soultanis \cite{MartiSoultanis:metric_fundamental_class}, under the assumption that there exists a \textit{metric fundamental class}. This assumption holds if $X$ is linearly locally contractible and has finite \textit{Nagata dimension} \cite{BassoMartiWenger:structures}.

\subsection{Analytic definition of quasiconformality}

The next theorem of Williams \cite{Williams:qc}*{Theorem\ 1.1 and Corollary\ 3.9} relates the geometric definitions of quasiconformality involving modulus with the analytic definition that relies on upper gradients. The final part of the statement is proved in \cite{NtalampekosRomney:length}*{Theorem 7.1}.

\begin{theorem}[Definitions of quasiconformality]\label{theorem:definitions_qc}
Let $X,Y$ be metric surfaces of locally finite Hausdorff $2$-measure, $h\colon X\to Y$ be a continuous map, and $K>0$. The following are equivalent.
\begin{enumerate}[label=\normalfont(\roman*)]
    \item\label{def:i} $h\in N^{1,2}_{\loc}(X,Y)$ and there exists a $2$-weak upper gradient $g$ of $h$ such that for every Borel set $E\subset Y$ we have
    $$\int_{h^{-1}(E)} g^2\, d\mathcal H^2 \leq K \mathcal H^2(E).$$
    \item\label{def:i'}Each point of $X$ has a neighborhood $U$ such that $h|_U\in N^{1,2}(U,Y)$ and there exists a $2$-weak upper gradient $g_U$ of $h|_U$ such that for every Borel set $E\subset Y$ we have
    $$\int_{(h|_{U})^{-1}(E)} g_U^2\, d\mathcal H^2 \leq K \mathcal H^2(E).$$
    \item\label{def:ii} For every curve family $\Gamma$ in $X$ we have
    $$\Mod_2 \Gamma \leq K\Mod_2 h(\Gamma).$$
\end{enumerate}
If, in addition, $h$ is monotone, then the above are equivalent to the following condition.
	\begin{enumerate}[label=\normalfont(\roman*)]\setcounter{enumi}{3}
	\item\label{def:iv} The set function $\nu(E)=\mathcal H^2(h(E))$ is an outer regular, locally finite Borel measure on $X$. Moreover, if $J_h$ is the Radon--Nikodym derivative of $\nu$ with respect to $\mathcal H^2$, then for $\mathcal H^2$-a.e.\ $x\in X$ we have
	$$g_h(x)^2\leq KJ_h(x).$$
	\end{enumerate}
\end{theorem}

A novel definition of quasiconformality in metric surfaces involving upper and \textit{lower} gradients is given in \cite{MeierRajala:definition}. It follows from \cite{MeierNtalampekos:rigidity}*{Lemma 2.19} that if there exists a weakly $K$-quasi\-conformal map from a Riemannian $2$-manifold $X$ onto a metric surface $Y$, then we necessarily have $K\geq 1$. It is unclear how to prove this for weakly quasiconformal maps between arbitrary surfaces.

\begin{question}
Let $h\colon X\to Y$ be a weakly $K$-quasiconformal map between metric surfaces. Do we necessarily have $K\geq 1$?
\end{question}

\subsection{Measure-theoretic properties of quasiconformal maps}

A consequence of a result of Rajala is the following theorem.

\begin{theorem}[\cite{Rajala:uniformization}*{Remark 8.3}]\label{theorem:rajala_lusin}
Let $f\colon X\to Y$ be a quasiconformal homeomorphism between a Riemannian surface $X$ and a metric surface $Y$ of locally finite Hausdorff $2$-measure. Then $f$ has the Lusin $(N^{-1})$ property. That is, if $E\subset X$ and $\mathcal H^2(E)>0$, then $\mathcal H^2(f(E))>0$. 
\end{theorem}

\begin{example}
Quasiconformal maps as in Theorem \ref{theorem:rajala} need not satisfy the Lusin $(N)$ property. Specifically, Rajala \cite{Rajala:uniformization}*{Proposition 17.1} gives an example of a quasiconformal map $f\colon \widehat \C\to Y$, where $Y$ is a metric surface embedded in $\R^3$, such that $\mathcal H^2(f(E))>0$ for some set $E\subset \widehat \C$ with $\mathcal H^2(E)=0$. Moreover, $Y$ can be taken to satisfy the upper mass bound $\mathcal H^2(B(y,r))\leq Cr^2$ for each $y\in Y$ and $0<r<\diam(Y)$.  
\end{example}

Tyson gave a sufficient condition for the Lusin $(N)$ property.

\begin{theorem}[\cite{Tyson:lusin}*{Corollary 5.10}]\label{theorem:tyson_lusin}
Let $f\colon X\to Y$ be a locally quasisymmetric homeomorphism between a Riemannian $n$-manifold $X$ and a metric space $Y$ of locally finite Hausdorff $n$-measure, where $n\geq 2$. Then $f$ has the Lusin $(N)$ property. That is, if $E\subset X$ and $\mathcal H^n(E)=0$, then $\mathcal H^n(f(E))=0$.
\end{theorem}

Combining Theorems \ref{theorem:rajala_lusin} and \ref{theorem:tyson_lusin} (see also \cite{Rajala:uniformization}*{Proposition 17.2}), we obtain the next result.

\begin{theorem}
Let $f\colon X\to Y$ be a quasiconformal and quasisymmetric homeomorphism between a Riemannian surface $X$ and a metric surface $Y$ of locally finite Hausdorff $2$-measure. Then $f$ has the Lusin $(N)$ and $(N^{-1})$ properties. That is, if $E\subset X$, then $\mathcal H^2(E)=0$ if and only if $\mathcal H^2(f(E))=0$. 
\end{theorem}

On the other hand, not every quasisymmetric map is quasiconformal in the setting of metric surfaces, as the next example shows.

\begin{example}\label{example:lusin_qs}
By \cites{Romney:absolute, NtalampekosRomney:absolute}, there exists a quasisymmetric map $f\colon \C\to Y$, where $Y$ has locally finite Hausdorff $2$-measure such that $\mathcal H^2(f(E))=0$ for some set $E\subset \C$ with $\mathcal H^2(E)>0$. That is, $f$ does not have the Lusin $(N^{-1})$ property. In view of Theorem \ref{theorem:rajala_lusin}, $f$ is not quasiconformal. However, by Theorem \ref{theorem:tyson}, $f$ is weakly quasiconformal. We remark that $Y$ cannot be reciprocal by Lemma \ref{lemma:upgrade}. 

Actually, the purpose of the examples in \cites{Romney:absolute, NtalampekosRomney:absolute} was to provide negative answers to the problem of inverse absolute continuity of quasisymmetric and quasiconformal maps. 
\end{example}

A metric space $X$ is \textit{$n$-rectifiable}, where $n\in \N$, if $X$ can be covered up to a set of Hausdorff $n$-measure zero by countably many Lipschitz images of subsets of $\R^n$. Not every metric surface of finite area is $2$-rectifiable. In fact, Rajala \cite{Rajala:uniformization}*{Proposition 17.1} provides an example of a reciprocal surface that is not rectifiable. However, the weakly quasiconformal uniformization theorem (Theorem \ref{theorem:wqc}) has the following consequence.

\begin{theorem}\label{theorem:rectifiable}
Let $X$ be a metric surface of locally finite Hausdorff $2$-measure. Then $X$ has a subset of positive Hausdorff $2$-measure that is $2$-rectifiable.
\end{theorem}

The result was first obtained in the case of locally geodesic surfaces by Meier and Wenger \cite{MeierWenger:uniformization}*{Corollary 1.5}. The conclusion also holds for any compact metric space of topological dimension $n$ and finite Hausdorff $n$-measure with positive \textit{lower density} almost everywhere \cite{DavidLeDonne:rectifiability}. We provide a short proof of Theorem \ref{theorem:rectifiable}, based on Theorem \ref{theorem:wqc}. 

\begin{proof}
Consider a Jordan region $Y\subset X$ of finite Hausdorff $2$-measure and a $(4/\pi)$-weakly quasiconformal map $f\colon \Omega \to Y$ as provided by Theorem \ref{theorem:wqc}, where $\Omega$ is $\C$ or $\D$. By the analytic definition of weak quasiconformality (see Theorem \ref{theorem:definitions_qc}) we have $f\in N^{1,2}(\Omega, Y)$. By \cite{LytchakWenger:discs}*{Proposition 3.2} or \cite{HeinonenKoskelaShanmugalingamTyson:Sobolev}*{Theorem 8.1.49}, there exist measurable sets $A_i\subset \Omega$, $i\in \N$, such that $f|_{A_i}$ is Lipschitz and $|\Omega \setminus \bigcup_{i\in \N}A_i|=0$. It suffices to show that $\mathcal H^2(f(A_i))>0$ for some $i\in \N$. Suppose that this not the case. By the analytic definition of weak quasiconformality, there exists a $2$-weak upper gradient $g$ of $f$ such that
$$\int_{f^{-1}(E)} g^2\, d\mathcal H^2 \leq \frac{4}{\pi} \mathcal H^2(E)$$
for each Borel set $E\subset Y$. This implies that $g=0$ a.e.\ in $A_i$ for each $i\in \N$. Therefore, $g=0$ a.e.\ in $\Omega$. The upper gradient inequality  and the continuity of $f$ imply that $f$ is constant on $\Omega$, a contradiction. 
\end{proof}

For manifolds admitting quasisymmetric parametrizations we have the next theorem, which follows from results of Tyson.

\begin{theorem}
Let $X\subset \R^n$, $n\geq 2$, be an open set, $Y$ be a metric space of locally finite Hausdorff $n$-measure, and $f\colon X\to Y$ be a locally quasisymmetric homeomorphism. Then $Y$ is $n$-rectifiable. 
\end{theorem}

\begin{proof}
By Theorem \ref{theorem:tyson}, $f$ is locally weakly quasiconformal. By  \cite{HeinonenKoskelaShanmugalingamTyson:Sobolev}*{Theorem 8.1.49}, as in the previous proof, $X$ can be decomposed up to a set $E$ of $n$-measure zero into countably many measurable sets in each of which $f$ is Lipschitz. By Theorem \ref{theorem:tyson_lusin}, $f$ has the Lusin $(N)$ property, so $\mathcal H^n(f(E))=0$, completing the proof.
\end{proof}

\bibliography{../../biblio} 

\end{document}